\documentclass[conference]{IEEEtran}
\usepackage{amsmath,amsthm,amssymb}
\newtheorem{theorem}{Theorem}
\usepackage{nicefrac}
\ifCLASSINFOpdf
   \usepackage[pdftex]{graphicx}
  \DeclareGraphicsExtensions{.pdf,.jpeg,.png}
\else
   \usepackage[dvips]{graphicx}
  \DeclareGraphicsExtensions{.eps}
\fi
\usepackage{epstopdf}
\usepackage{tikz}
\usetikzlibrary{calc,positioning}
\usepackage{algpseudocode}
\usepackage{algorithm}
\usepackage{amsfonts}

\usepackage{fixltx2e}

\usepackage{stfloats}
\usepackage{url}

\begin{document}
%
\title{Stochastic Linearization of Multivariate Nonlinearities}

\author{\IEEEauthorblockN{Sarnaduti Brahma and Hamid R. Ossareh}
\IEEEauthorblockA{Department of Electrical and Biomedical Engineering\\
The University of Vermont\\
Burlington, Vermont 05401\\
Email: \url{{sbrahma,hossareh}@uvm.edu}}}


%


\maketitle

\begin{abstract}
Stochastic linearization is a method used in Quasilinear Control (QLC) to replace a nonlinearity by an equivalent gain and a bias, utilizing the statistical properties of random inputs. In this paper, the theory of stochastic linearization is extended to nonlinear functions of multiple variables or inputs forming a multivariate Gaussian vector. The result is applied to find the stochastic linearization of a bivariate saturation nonlinearity in a general feedback control system. The accuracy of stochastic linearization has been investigated by a Monte Carlo simulation and has been found out to be fairly high. Finally, a practical example of optimal control design using QLC is presented.
\end{abstract}


%
\IEEEpeerreviewmaketitle

\section{Introduction}
Actuators and sensors in control systems are often nonlinear. While plants are generally nonlinear as well, they can usually be linearized around an operating point if the control system is well designed. The nonlinear instrumentation, i.e., the actuators and sensors, however, cannot. This is because external random inputs to the system may force them to operate far from their designed operating point, activating nonlinearities in them.

Quasilinear Control (QLC) is a set of methods that can be used to analyze and design control systems with nonlinear actuators and sensors \cite{Ching2010}. It leverages the method of stochastic linearization, which replaces each nonlinearity by an equivalent gain and a bias, based on statistical properties of the stochastic inputs. Consider a nonlinear function $f\left(x\right)=x^2$, as shown in Fig. \ref{fig:The-Approach-of}. The traditional (Jacobian) approach in  linearizing such a function is to find the derivative of the function at a suitable operating point, replace the nonlinearity by a linear approximation and shift the origin at that operating point. The method of stochastic linearization, on the other hand, is based on minimizing the expected value of the mean squared error between the nonlinear function and its stochastically linearized approximation, taking into account the probability distribution of the input. This approach is graphically depicted in Fig. \ref{fig:The-Approach-of}, in which the input $x$ has been assumed to follow a Normal distribution with mean 3 and standard deviation 2. The dashed lines represent two Jacobian linearizations performed at (1,1) and (4,16) and the dotted line represents the stochastically linearized approximation. If Jacobian linearization is performed at (1,1), but the operating point shifts to, say, (4,16), the linearization becomes highly inaccurate. However, since stochastic linearization considers the expected value of the derivatives around $x=3$, it performs better at (4,16). 

\begin{figure}[t]
	\begin{centering}
		\includegraphics[width=0.9\columnwidth]{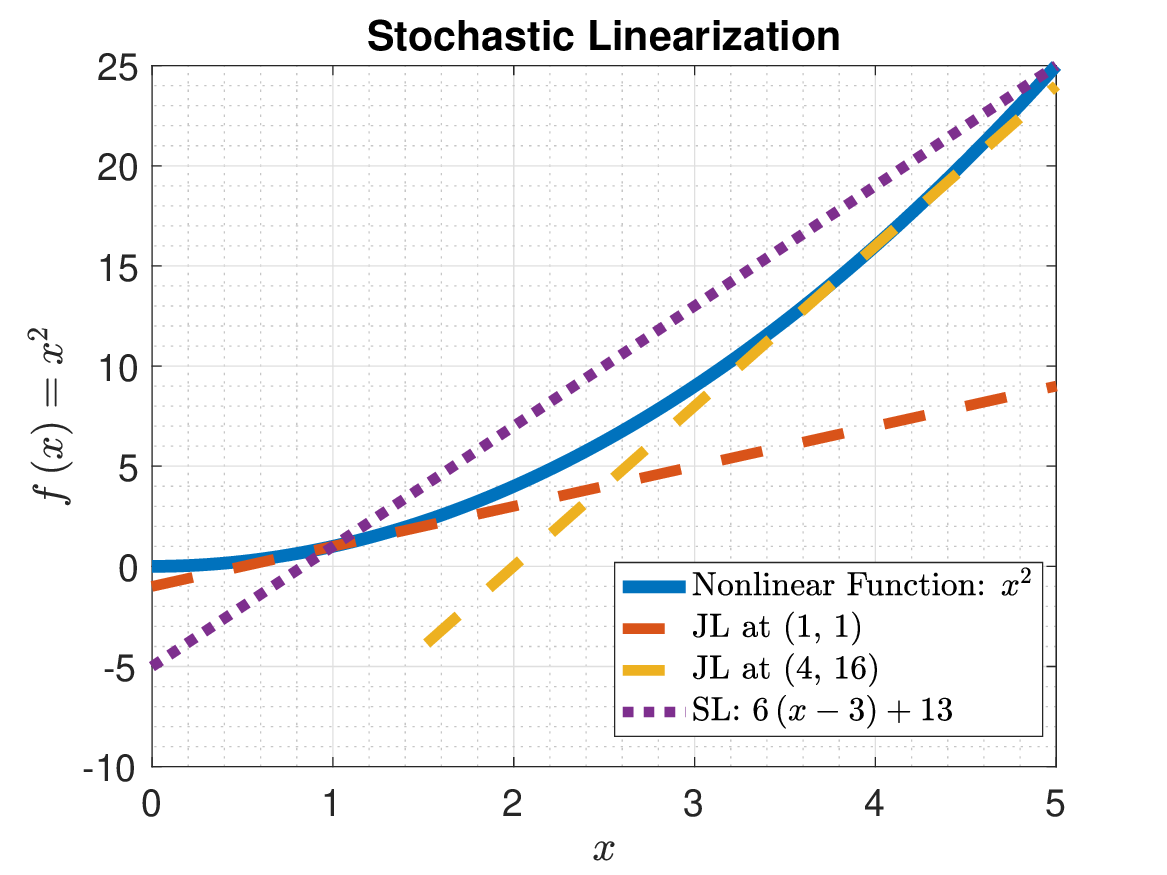}
		\par\end{centering}
	\caption{The Approach of Stochastic Linearization\label{fig:The-Approach-of}}
\end{figure}

The theory of stochastic linearization and quasilinear control has been developed only for nonlinear actuators and sensors having a single input. In practical applications, however, more than just one factor or input affect the performance of a nonlinear actuator, and hence the operation of the control system. As a case in point, the authors are involved in a US Department of Energy Project called ENERGIZE, whose goal is to develop robust and resilient real-time control systems with uncertain distributed energy resources. In such a renewable energy system, that involves aggregation of several distributed energy resources that can both produce and consume electric power, it is often necessary to compute the optimal power set point for these aggregated devices to ensure a robust and resilient operation. It is then desirable to estimate the power limits of such resources. Since the power limits would depend on the number and type of devices being aggregated, both of which are random phenomena, depending on when users decide to turn the devices on or off, the saturation authority of the actuator can be considered to be a stochastic process in such a case.

In this paper, the theory of stochastic linearization is extended to a nonlinear function of multiple variables, such that the inputs form a wide sense stationary (WSS) multivariate Gaussian random vector. The outline is as follows: Section \ref{rsv} provides a brief review of single variable QLC. Section \ref{msl} introduces expressions for stochastically linearizing a generic multivariate nonlinearity. In Section \ref{abs}, the bivariate saturation nonlinearity is introduced and expressions for its equivalent gains and bias derived, using the result of the previous Section. The result is used in Section \ref{qlccl} to find the stochastic linearization of a general feedback control system, in which the actuator is a bivariate saturation, with the randomness in the bounds modeled as a second input to the actuator, taking the reference and disturbance signals to be WSS Gaussian random processes with specified means and standard deviations. Section \ref{cSQL} explores special features of multi-variable QLC compared to single-variable QLC, specifically the effect of the second actuator input and correlation between the actuator inputs. Section \ref{asl} investigates the accuracy of stochastic linearization by a Monte Carlo simulation with different possible input and system parameters. In Section \ref{pe}, a practical example of an optimal controller design has been provided. Section \ref{conc} concludes the paper. In the appendix, series expansions of the integrals in Section \ref{qlccl} have been derived, along with their region of convergence, and an algorithm presented for their calculation.

\section{Review of Single Variable Stochastic Linearization}\label{rsv}
This section presents a brief review of single variable stochastic linearization. For details, please refer to \cite{Ching2010}.

\subsection{Open Loop System}\label{ols}
Consider a single input single output (SISO) system shown in Fig. \ref{fig:p2} driven by a wide-sense stationary Gaussian stochastic input $u(t)$, such that it is governed by the input-output relationship:
\[v(t)=f(u(t))\]
Stochastic linearization replaces the above nonlinearity by a linear approximation $Nu_0(t)+M$, such that the functional:
\[
\epsilon\left(N,M\right)=E\left\{ \left[f\left(u\left(t\right)\right)-Nu_0\left(t\right)-M\right]^{2}\right\} 
\] is minimized \cite{Kabamba2015}. Here $N$ is called the \textit{quasilinear gain}, $M$ the \textit{quasilinear bias} and $u_0(t)$ is the zero-mean part of $u(t)$. It can be shown that the values of $N$ and $M$ are:
\begin{equation}\label{neqs}
N=E[f'(u)]
\end{equation}
\begin{equation}\label{meqs}
M=E[f(u)]
\end{equation}
\begin{figure}
	\centering
	\includegraphics[width=1\linewidth]{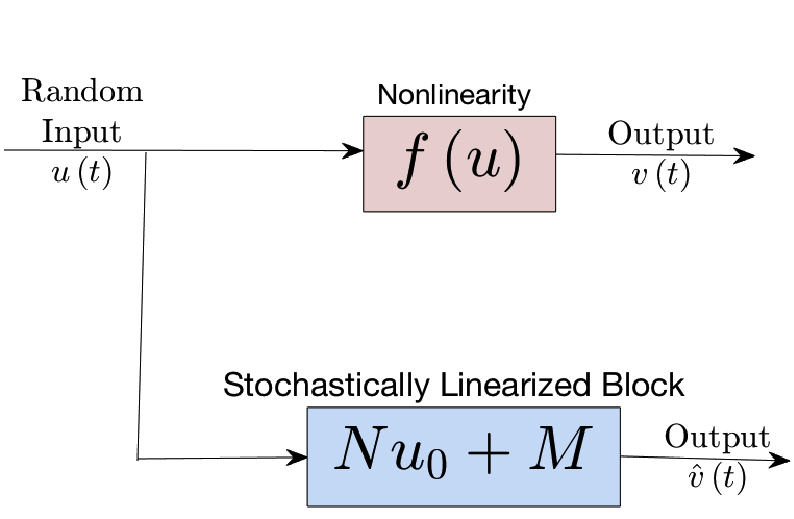}
	\caption{The process of Single Variable Stochastic Linearization: Here $N$ is the quasilinear gain defined in \eqref{neqs} and $M$ the quasilinear bias defined in \eqref{meqs} and $u_0(t)$ is the zero-mean part of $u(t)$.}
	\label{fig:p2}
\end{figure}

\subsection{Closed Loop System}\label{clssv}
Consider a general feedback control system shown in Fig. \ref{fig:Block-Diagram-ofs}. It consists of a plant $P\left(s\right)$ whose output is desired to be controlled using a controller $C\left(s\right)$ and a nonlinear actuator described by the function $f\left(\cdot\right)$. The signal $r(t)$ is the reference signal to be tracked and $d(t)$ is the disturbance. $r(t)$ is generated by passing white noise $w_r(t)$ through a coloring filter $F_{\Omega_r}(s)$, scaling it by $\sigma_r$, and adding a bias $\mu_r$. Similarly, $d(t)$ is generated by passing white noise $w_d(t)$ through a coloring filter $F_{\Omega_d}(s)$, scaling it by $\sigma_d$, and adding a bias $\mu_d$. The $H_2$-norm of the coloring filters are considered to be unity, which allows the reference and disturbance signals to have means $\mu_r$ and $\mu_d$ respectively, along with corresponding standard deviations $\sigma_r$ and $\sigma_d$. The block diagram shows the corresponding state space representations. The nonlinearity, $f(\cdot)$ can be replaced by a stochastically linearized block $Nu_0(t)+M$, where $u_0(t)$ is the zero-mean part of $u(t)$ and $N$ and $M$ are as in \eqref{neqs} and \eqref{meqs} respectively. To calculate $N$ and $M$, we need to consider the entire system, along with the statistics of the input signals. 

This is the general idea behind stochastically linearizing a nonlinear control system with a single input to the nonlinearity. Since it is a special case of multivariate stochastic linearization, which will be described in detail in Section \ref{qlccl}, we will not review its details here.

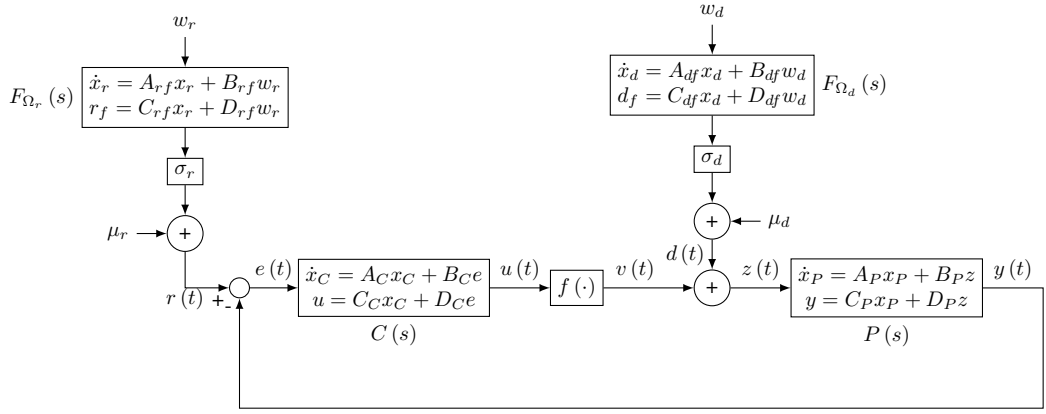
\begin{figure*}[t]
	\centering
	\begin{tikzpicture}[>=latex,scale=0.8,every node/.style={transform shape}, every text node part/.style={align=center}]
	\node (wr) at (0,0){$w_r$};
	\node[rectangle,draw] (rf) [below=0.5 of wr,label={left:$F_{\Omega_r}\left(s\right)$}]{$\dot{x}_r=A_{rf}x_r+B_{rf}w_r$\\$r_f=C_{rf}x_r+D_{rf}w_r$};
	\node[rectangle,draw] (sr) [below=0.5 of rf]{$\sigma_r$};
	\node[circle,draw] (cr) [below=0.5 of sr]{+};
	\node (mur) [left=0.5 of cr]{$\mu_r$};
	\node[circle,draw] (sum) [below right=0.8 of cr]{};
	\node[rectangle,draw] (C) [right=0.8 of sum,label={below:$C\left(s\right)$}]{$\dot{x}_C=A_{C}x_C+B_{C}e$\\$u=C_{C}x_C+D_Ce$};
	\node[rectangle,draw] (sat) [right= of C]{$f\left(\cdot\right)$};
	\node[circle,draw] (dsum) [right=1.5 of sat]{+};

	\node[circle,draw] (cd) [above=0.5 of dsum]{+};
	\node (mud) [right=0.5 of cd]{$\mu_{d}$};
	\node[rectangle,draw] (sd) [above=0.5 of cd]{$\sigma_{d}$};
	\node[rectangle,draw] (df) [above=0.5 of sd,label={right:$F_{\Omega_d}\left(s\right)$}]{$\dot{x}_{d}=A_{df}x_{d}+B_{df}w_{d}$\\${d}_f=C_{df}x_{d}+D_{df}w_{d}$};
	\node (wd) [above=0.5 of df]{$w_{d}$};
	
	\node[rectangle,draw] (P) [right= of dsum,label={below:$P\left(s\right)$}]{$\dot{x}_P=A_{P}x_P+B_{P}z$\\$y=C_{P}x_P+D_Pz$};
	\node[coordinate] (y) [right= of P]{};
	
	\draw[->] (P)--node[above]{$y\left(t\right)$}(y)--($(y)+(0,-2)$)-|node[pos=0.95,left]{-}(sum);
	\draw[->] (wr)--(rf);
	\draw[->] (rf)--(sr);
	\draw[->] (sr)--(cr);
	\draw[->] (mur)--(cr);
	\draw[->] (cr)|-node[pos=0.4,below]{$r\left(t\right)$}node[very near end,below]{+}(sum);
	\draw[->] (wd)--(df);
	\draw[->] (df)--(sd);
	\draw[->] (sd)--(cd);
	\draw[->] (mud)--(cd);
	\draw[->] (cd)--node[left]{$d\left(t\right)$}(dsum);
	\draw[->] (sum)--node[above]{$e\left(t\right)$}(C);
	\draw[->] (C)--node[above]{$u\left(t\right)$}(sat);
	\draw[->] (sat)--node[above,pos=0.35]{$v\left(t\right)$}(dsum);
	\draw[->] (dsum)--node[above]{$z\left(t\right)$}(P);
	\end{tikzpicture}
	\caption{Block Diagram of Control System\label{fig:Block-Diagram-ofs}}
    \label{figsv}
\end{figure*}

\section{Multivariate Stochastic Linearization}\label{msl}
Consider a multi-input single output (MISO) system driven by $n$ wide-sense stationary Gaussian inputs $u_1(t),u_2(t),\dots,u_n(t)$, forming a Gaussian random vector $\mathbf{u}(t)$, and modeled by a multivariate nonlinearity $v(t)=f(\mathbf{u}(t))$, where $v(t)$ is the output. The problem is to find a linear approximation to this nonlinearity. Various objective functions have been suggested in the literature for minimizing the error introduced by a possible linear approximation \cite{Iwan1972,Smith1966}, but it has been found that, in general, the mean squared error between the nonlinearity and its linear approximation gives results as good as, if not better, than others \cite{Spanos}. Hence, it is used in the following Theorem to derive the linear approximation.

\begin{theorem}
	Let $u_{1}\left(t\right)$,$u_{2}\left(t\right)$,$\ldots$,$u_{n}\left(t\right)$
			be $n$ WSS jointly Gaussian processes with expected values $\mu_{1}\left(t\right),\mu_{2}\left(t\right),\ldots,\mu_{n}\left(t\right)$
			respectively, $\mathbf{u}\left(t\right)=\left[\begin{array}{cccc}
			u_{1}\left(t\right) & u_{2}\left(t\right) & \ldots & u_{n}\left(t\right)\end{array}\right]^{T}$, $\boldsymbol{\mu}\left(t\right)=~\left[\begin{array}{cccc}
			\mu_{1}\left(t\right) & \mu_{2}\left(t\right) & \ldots & \mu_{n}\left(t\right)\end{array}\right]^{T}$ and $\mathbf{u_{0}}\left(t\right)=~\mathbf{u}\left(t\right)-\boldsymbol{\mu}\left(t\right)$.
			Then, for any piecewise differentiable function $f\left(\mathbf{u}\right):\mathbb{R}^{n}\rightarrow\mathbb{R}$,
			the functional
			\[
			\epsilon\left(\mathbf{N},M\right)=E\left\{ \left[f\left(\mathbf{u}\left(t\right)\right)-\mathbf{N^{T}}\mathbf{u_{0}}\left(t\right)-M\right]^{2}\right\} 
			\]
			is minimized by:
			\begin{equation}\label{neq}
			\mathbf{N}=E\left[\mathbf{\boldsymbol{\nabla}}f\left(\mathbf{u}\left(t\right)\right)\right]
			\end{equation}
			\begin{equation}\label{meqd}
			M=E\left[f\left(\mathbf{u}\left(t\right)\right)\right]
			\end{equation}
		where $\mathbf{N}=\left[\begin{array}{cccc}
			N_{1} & N_{2} & \ldots & N_{n}\end{array}\right]^{T}$ is a constant vector.
\end{theorem}

\begin{proof}
	To minimize $\epsilon\left(\mathbf{N},M\right)$, we set $\frac{\partial\epsilon\left(\mathbf{N},M\right)}{\partial\mathbf{N}}=\mathbf{0}$
	and $\frac{\partial\epsilon\left(\mathbf{N},M\right)}{\partial M}=0$. The term $\frac{\partial\epsilon\left(\mathbf{N,M}\right)}{\partial\mathbf{N}}$
	is calculated as follows. Let $\mathbf{u_{0}}\left(t\right)=\left[\begin{array}{cccc}
	u_{01}\left(t\right) & u_{02}\left(t\right) & \ldots & u_{0n}\left(t\right)\end{array}\right]^{T}$. Then for any $k\in\mathbb{N},\,1\leq k\leq n$,
	
		\[\frac{\partial\epsilon\left(\mathbf{N},M\right)}{\partial N_{k}}\]
   {\small \begin{eqnarray*}
		& = & E\left\{ \frac{\partial}{\partial N_{k}}\left[f\left(\mathbf{u}\left(t\right)\right)-\mathbf{N^{T}}\mathbf{u_{0}}\left(t\right)-M\right]^{2}\right\} \\
		& = & E\left\{ 2\left[f\left(\mathbf{u}\left(t\right)\right)-\mathbf{N^{T}u_{0}}\left(t\right)-M\right]\left[-\frac{\partial}{\partial N_{k}}\left(\mathbf{N^{T}u_{0}}\left(t\right)\right)\right]\right\} \\
		& = & E\left\{ 2\left[f\left(\mathbf{u}\left(t\right)\right)-\mathbf{N^{T}u_{0}}\left(t\right)-M\right]\left[-\frac{\partial}{\partial N_{k}}\sum_{i=1}^{n}N_{i}u_{0i}\left(t\right)\right]\right\} \\
		& = & E\left\{ 2\left[f\left(\mathbf{u}\left(t\right)\right)-\mathbf{N^{T}u_{0}}\left(t\right)-M\right]\left[-u_{0k}\left(t\right)\right]\right\} 
	\end{eqnarray*}}
\[\therefore\frac{\partial\epsilon\left(\mathbf{N},M\right)}{\partial N_{k}}\]
	{
		\begin{eqnarray*}
			& = & -2E\left[u_{0k}\left(t\right)f\left(\mathbf{u}\left(t\right)\right)\right]+2E\left[u_{0k}\left(t\right)\mathbf{N^{T}u_{0}}\left(t\right)\right]\\
			& & +2E\left[u_{0k}\left(t\right)M\right]\\
			& = & -2E\left[u_{0k}\left(t\right)f\left(\mathbf{u}\left(t\right)\right)\right]+2E\left[u_{0k}\left(t\right)\sum_{i=1}^{n}N_{i}u_{0i}\left(t\right)\right]\\
			& & +2E\left[u_{0k}\left(t\right)\right]M\\
			& = & -2E\left[u_{0k}\left(t\right)f\left(\mathbf{u}\left(t\right)\right)\right]+2\sum_{i=1}^{n}N_{i}E\left[u_{0k}\left(t\right)u_{0i}\left(t\right)\right]
		\end{eqnarray*}
	}
	
	as clearly, $u_{0k}\left(t\right)=u_{k}\left(t\right)-\mu_{k}\left(t\right)$
	are zero-mean processes.
	\[\therefore\frac{\partial\epsilon\left(\mathbf{N},M\right)}{\partial\mathbf{N}}\]
	\begin{eqnarray*}
		 = \left[\begin{array}{c}
		 	\frac{\partial\epsilon\left(\mathbf{N},M\right)}{\partial N_{1}}\\
		 	\frac{\partial\epsilon\left(\mathbf{N},M\right)}{\partial N_{2}}\\
		 	\vdots\\
		 	\frac{\partial\epsilon\left(\mathbf{N},M\right)}{\partial N_{n}}
		 \end{array}\right]&  = & -2E\left(\left[\begin{array}{c}
			u_{01}\left(t\right)f\left(\mathbf{u}\left(t\right)\right)\\
			u_{02}\left(t\right)f\left(\mathbf{u}\left(t\right)\right)\\
			\vdots\\
			u_{0n}\left(t\right)f\left(\mathbf{u}\left(t\right)\right)
		\end{array}\right]\right)\\
	&  & +2\textrm{cov}\left[\mathbf{u_0}\left(t\right)\right]\left[\begin{array}{c}
		N_{1}\\
		N_{2}\\
		\vdots\\
		N_{n}
	\end{array}\right]
	\end{eqnarray*}
	where 
	\[\textrm{cov}\left[\mathbf{u_0}\left(t\right)\right]=E\left[\mathbf{u_{0}}\left(t\right)\mathbf{u_{0}^{T}}\left(t\right)\right]\]
	{ \scriptsize
		\[
		=\left[\begin{array}{cccc}
		\sigma_{u_{01}}^{2} & E\left[u_{01}\left(t\right)u_{02}\left(t\right)\right] & \cdots & E\left[u_{01}\left(t\right)u_{0n}\left(t\right)\right]\\
		E\left[u_{02}\left(t\right)u_{01}\left(t\right)\right] & \sigma_{u_{02}}^{2} & \cdots & E\left[u_{02}\left(t\right)u_{0n}\left(t\right)\right]\\
		\vdots & \vdots & \ddots & \vdots\\
		E\left[u_{0n}\left(t\right)u_{01}\left(t\right)\right] & E\left[u_{0n}\left(t\right)u_{02}\left(t\right)\right] & \cdots & \sigma_{u_{0n}}^{2}
		\end{array}\right]
		\]
	}
is the covariance matrix of $\mathbf{u_{0}}\left(t\right)$.
    \[\therefore\frac{\partial\epsilon\left(\mathbf{N},M\right)}{\partial\mathbf{N}} = -2E\left[f\left(\mathbf{u}\left(t\right)\right)\mathbf{u_{0}}\left(t\right)\right]+2\textrm{cov}\left[\mathbf{u_{0}}\left(t\right)\right]\mathbf{N}\]

	The term $E\left[f\left(\mathbf{u_{0}}\left(t\right)\right)\mathbf{u_{0}}\left(t\right)\right]$
	can be expanded using the following result from \cite{Kazakov}:
		\[
		E\left[g\left(\boldsymbol{\eta}\right)\boldsymbol{\eta}\right]=E\left[\boldsymbol{\eta\eta^{T}}\right]E\left[\boldsymbol{\nabla}g\left(\boldsymbol{\eta}\right)\right]
		\] where $\boldsymbol{\eta}$ is any $n\times1$ jointly Gaussian vector and
	\[\boldsymbol{\nabla}=\left[\begin{array}{cccc}
	\frac{\partial}{\partial\eta_{1}} & \frac{\partial}{\partial\eta_{2}} & \cdots & \frac{\partial}{\partial\eta_{n}}\end{array}\right]^{\mathbf{T}}\] is the gradient operator.
Hence,
\[\therefore\frac{\partial\epsilon\left(\mathbf{N},M\right)}{\partial\mathbf{N}}\]
	\begin{eqnarray*}
		 & = & -2E\left[\mathbf{u_{0}}\left(t\right)\mathbf{u_{0}^{T}}\left(t\right)\right]E\left[\boldsymbol{\nabla}f\left(\mathbf{u}\left(t\right)\right)\right]+2\textrm{cov}\left[\mathbf{u_{0}}\left(t\right)\right]\mathbf{N}\\
		& = & -2\textrm{cov}\left[\mathbf{u_{0}}\left(t\right)\right]E\left[\boldsymbol{\nabla}f\left(\mathbf{u}\left(t\right)\right)\right]+2\textrm{cov}\left[\mathbf{u_{0}}\left(t\right)\right]\mathbf{N}\\
		& = & 2\textrm{cov}\left[\mathbf{u_{0}}\left(t\right)\right]\left\{ \mathbf{N}-E\left[\boldsymbol{\nabla}f\left(\mathbf{u}\left(t\right)\right)\right]\right\} 
	\end{eqnarray*}

$\epsilon\left(\mathbf{N},M\right)$ is minimized when $\frac{\partial\epsilon\left(\mathbf{N},M\right)}{\partial\mathbf{N}}=\mathbf{0}$,
i.e. when $\mathbf{N}=~E\left[\boldsymbol{\nabla}f\left(\mathbf{u}\left(t\right)\right)\right]$,
since $\textrm{cov}\left[\mathbf{u_{0}}\left(t\right)\right]$ is
positive definite.

The term $\frac{\partial\epsilon\left(\mathbf{N},M\right)}{\partial M}$ is calculated as follows.
{\small\begin{eqnarray*}
	\frac{\partial\epsilon\left(\mathbf{N},M\right)}{\partial M} & = & E\left\{ \frac{\partial}{\partial M}\left[f\left(\mathbf{u}\left(t\right)\right)-\mathbf{N^{T}}\mathbf{u_{0}}\left(t\right)-M\right]^{2}\right\} \\
	& = & E\left\{ 2\left[-f\left(\mathbf{u}\left(t\right)\right)+\mathbf{N^{T}u_{0}}\left(t\right)+M\right]\right\} \\
	& = & -2E\left[f\left(\mathbf{u}\left(t\right)\right)\right]+2E\left[\mathbf{N^{T}u_{0}}\left(t\right)\right]+2E\left[M\right]\\
	& = & -2E\left[f\left(\mathbf{u}\left(t\right)\right)\right]+2E\left[\sum_{i=1}^{n}N_{i}u_{0i}\left(t\right)\right]+2M\\
	& = & -2E\left[f\left(\mathbf{u}\left(t\right)\right)\right]+2\sum_{i=1}^{n}N_{i}E\left[u_{0i}\left(t\right)\right]+2M\\
	& = & -2E\left[f\left(\mathbf{u}\left(t\right)\right)\right]+2M\\
	& = & 2\left\{ M-E\left[f\left(\mathbf{u}\left(t\right)\right)\right]\right\} 
\end{eqnarray*}}

$\epsilon\left(\mathbf{N},M\right)$ is minimized when $\frac{\partial\epsilon\left(\mathbf{N},M\right)}{\partial M}=0$,
i.e. when $M=~E\left[f\left(\mathbf{u}\left(t\right)\right)\right]$.
This completes the proof.
	\end{proof}
    
Equations \eqref{neq} and \eqref{meqd} are similar to \eqref{neqs} and \eqref{meqs} respectively, with some differences. There are now $n$ quasilinear \textit{gains} forming a vector $\mathbf{N}$. Also, the single input $u(t)$ in \eqref{neqs} and \eqref{meqs} is replaced by a multiple-input vector $\mathbf{u}(t)$, and the derivative in \eqref{neqs} replaced by a gradient in \eqref{neq}. The process is illustrated in Fig. \ref{fig:p4}.

\begin{figure}
	\centering
	\includegraphics[width=1\linewidth]{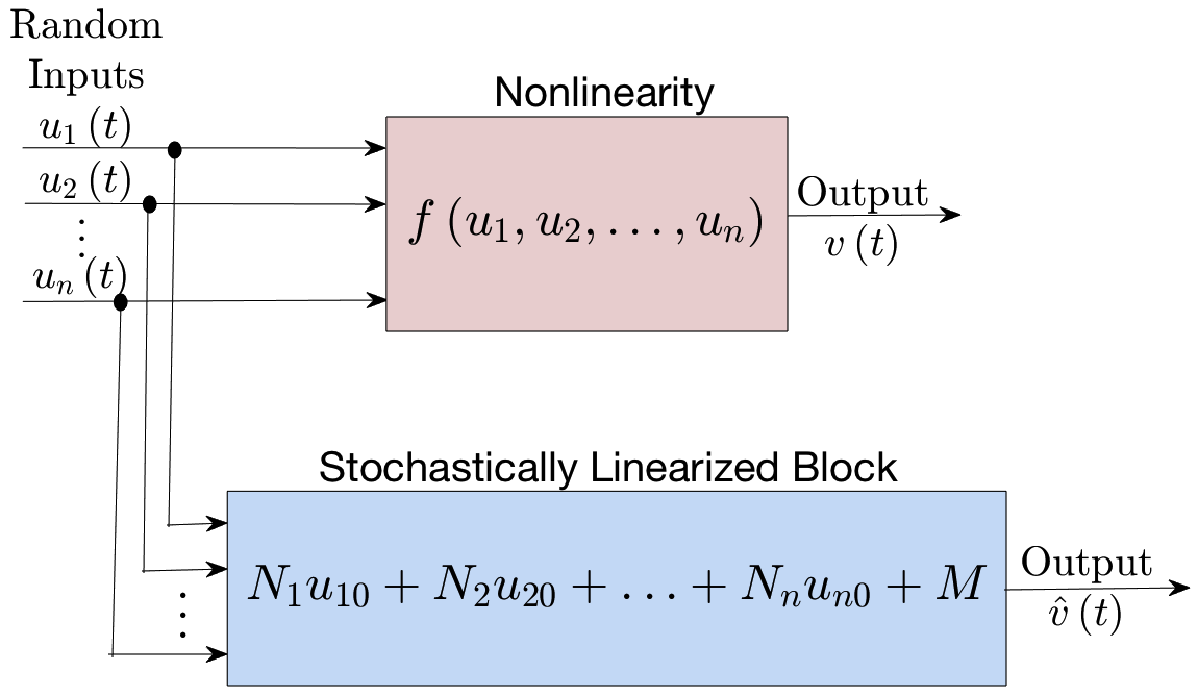}
	\caption{The process of Multi-variable Stochastic Linearization. Here $N_1,N_2,\dots,N_n$ are the quasilinear gains forming the vector $\mathbf{N}$ defined in \eqref{neq} and $M$ is the quasilinear bias defined in \eqref{meq}. $u_{01},u_{02},\dots,u_{0n}$ are the zero-mean parts of $u_1,u_2,\dots,u_n$ respectively.}
	\label{fig:p4}
\end{figure}

\section{Application to Bivariate Saturation}\label{abs}
In this section, the method of multivariate stochastic linearization is applied to the bivariate saturation nonlinearity, which is illustrated in Fig. \ref{bsn} and defined as follows:
\[\textrm{sat}_{\alpha,\beta}\left[u_{1}\left(t\right),u_{2}\left(t\right)\right]
\]
{\scriptsize
	\begin{equation}\label{bsat}
	=\begin{cases}
	\begin{cases}
	\beta+u_{2}\left(t\right), & u_{1}\left(t\right)>\beta+u_{2}\left(t\right)\\
	u_{1}\left(t\right), & \alpha-u_{2}\left(t\right)\leq u_{1}\left(t\right)\leq\beta+u_{2}\left(t\right)\\
	\alpha-u_{2}\left(t\right), & u_{1}\left(t\right)<\alpha-u_{2}\left(t\right)
	\end{cases} & ,u_{2}\left(t\right)\geq\textrm{max}\left(-\beta,\alpha\right)\\
	0 & ,u_{2}\left(t\right)<\textrm{max}\left(-\beta,\alpha\right)
	\end{cases}
	\end{equation}
}

We call $u_1(t)$ the primary input and $u_2(t)$ the secondary input. On substituting the nonlinear function \eqref{bsat} for $f(\cdot)$ in \eqref{neq} and \eqref{meqd}, the values of $N_1$ and $N_2$ can be found out to be:

\begin{figure}[b]
	\centering
	{\footnotesize{}\begin{tikzpicture}[every node/.style={transform shape}, >=latex]
		\draw [->,thick](-2.5,0)--(2,0) node[anchor=west] {$u_1\left(t\right)$};
		\draw [->,thick](0,-2)--(0,1.5) node[anchor=south] {$\textrm{sat}_{\alpha}\left[u_1\left(t\right),u_2\left(t\right)\right]$};
		\draw[thick] (-2pt,1) node[anchor=east]{$\beta$}--(2pt,1);
		\draw[thick] (-2pt,-1.5) node[anchor=east]{$\alpha$}--(2pt,-1.5);
		\draw[thick] (-1.5,-2pt) node[anchor=south]{$\alpha-u_2\left(t\right)$}--(-1.5,2pt);
		\draw[thick] (1,2pt) node[anchor=north]{$\beta+u_2\left(t\right)$}--(1,-2pt);
		\draw [color=gray](-2.5,-1.5)--(-1.5,-1.5);
		\draw (-1.5,-1.5)--(1,1);
		\draw [color=gray](1,1)--(2,1);
		\draw[color=blue, samples=40] plot[domain=1:2] (\x,1+rand*0.1) node[anchor=west] {$u_2 \left(t\right)$};
		\draw[color=blue, samples=40] plot[domain=-2.5:-1.5] (\x,-1.5-rand*0.1) (-2.5,-1.5)node[anchor=east] {$-u_2 \left(t\right)$};
		\end{tikzpicture}}{\footnotesize \par}
	{\footnotesize{}\caption{Bivariate Saturation Nonlinearity}\label{bsn}
	}{\footnotesize \par}
\end{figure}
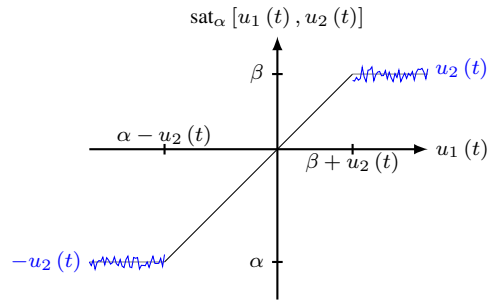

\begin{gather*}
\mathbf{N}=\left[\begin{array}{cc}
N_{1} & N_{2}\end{array}\right]^{T}	=	E\left\{ \boldsymbol{\nabla}\textrm{sat}_{\alpha,\beta}\left[u_{1}\left(t\right),u_{2}\left(t\right)\right]\right\} \\
=	E\left\{ \left[\begin{array}{c}
\frac{\partial}{\partial u_{1}}\textrm{sat}_{\alpha,\beta}\left[u_{1}\left(t\right),u_{2}\left(t\right)\right]\\
\frac{\partial}{\partial u_{2}}\textrm{sat}_{\alpha,\beta}\left[u_{1}\left(t\right),u_{2}\left(t\right)\right]
\end{array}\right]\right\} \\
{\tiny =\begin{cases}
E\left[\begin{array}{c}
\begin{cases}
0, & u_{1}\left(t\right)>\beta+u_{2}\left(t\right)\\
1, & \alpha-u_{2}\left(t\right)\leq u_{1}\left(t\right)\leq\beta+u_{2}\left(t\right)\\
0, & u_{1}\left(t\right)<\alpha-u_{2}\left(t\right)
\end{cases}\\
\begin{cases}
1, & u_{1}\left(t\right)>\beta+u_{2}\left(t\right)\\
0, & \alpha-u_{2}\left(t\right)\leq u_{1}\left(t\right)\leq\beta+u_{2}\left(t\right)\\
-1 & u_{1}\left(t\right)<\alpha-u_{2}\left(t\right)
\end{cases}
\end{array}\right] & ,u_{2}\left(t\right)\geq\textrm{max}\left(-\beta,\alpha\right)\\
0 & ,u_{2}\left(t\right)<\textrm{max}\left(-\beta,\alpha\right)
\end{cases}}
\end{gather*}
or re-written as in \eqref{n1e1} and \eqref{n2e1}.

\begin{figure*}[t]
	\centering
	\begin{tikzpicture}[>=latex,scale=0.9,every node/.style={transform shape}, every text node part/.style={align=center}]
	\node (wr) at (0,0){$w_r$};
	\node[rectangle,draw] (rf) [below=0.5 of wr,label={left:$F_{\Omega_r}\left(s\right)$}]{$\dot{x}_r=A_{rf}x_r+B_{rf}w_r$\\$r_f=C_{rf}x_r+D_{rf}w_r$};
	\node[rectangle,draw] (sr) [below=0.5 of rf]{$\sigma_r$};
	\node[circle,draw] (cr) [below=0.5 of sr]{+};
	\node (mur) [left=0.5 of cr]{$\mu_r$};
	\node[circle,draw] (sum) [below right=0.8 of cr]{};
	\node[rectangle,draw] (C) [right=0.8 of sum,label={below:$C\left(s\right)$}]{$\dot{x}_C=A_{C}x_C+B_{C}e$\\$u=C_{C}x_C+D_Ce$};
	\node[rectangle,draw] (sat) [right= of C]{$\textrm{sat}_{\alpha}^{\beta}\left(u,{\beta}_n\right)$};
	\node[circle,draw] (dsum) [right=3 of sat]{+};
	
	\node[circle,draw] (cb) [above=0.5 of sat]{+};
	\node (mub) [left=0.5 of cb]{$\mu_{\beta}$};
	\node[rectangle,draw] (sb) [above=0.5 of cb]{$\sigma_{\beta}$};
	\node[rectangle,draw] (bf) [above=0.5 of sb,label={left:$F_{\Omega_{\beta}}\left(s\right)$}]{$\dot{x}_{\beta}=A_{{\beta}f}x_{\beta}+B_{{\beta}f}w_{\beta}$\\${\beta}_f=C_{\beta f}x_{\beta}+D_{\beta f}w_{\beta}$};
	\node (wb) [above=0.5 of bf]{$w_{\beta}$};
	
	\node[circle,draw] (cd) [above=0.5 of dsum]{+};
	\node (mud) [right=0.5 of cd]{$\mu_{d}$};
	\node[rectangle,draw] (sd) [above=0.5 of cd]{$\sigma_{d}$};
	\node[rectangle,draw] (df) [above=0.5 of sd,label={right:$F_{\Omega_d}\left(s\right)$}]{$\dot{x}_{d}=A_{df}x_{d}+B_{df}w_{d}$\\${d}_f=C_{df}x_{d}+D_{df}w_{d}$};
	\node (wd) [above=0.5 of df]{$w_{d}$};
	
	\node[rectangle,draw] (P) [right= of dsum,label={below:$P\left(s\right)$}]{$\dot{x}_P=A_{P}x_P+B_{P}z$\\$y=C_{P}x_P+D_Pz$};
	\node[coordinate] (y) [right= of P]{};
	
	\draw[->] (P)--node[above]{$y\left(t\right)$}(y)--($(y)+(0,-2)$)-|node[pos=0.95,left]{-}(sum);
	\draw[->] (wr)--(rf);
	\draw[->] (rf)--(sr);
	\draw[->] (sr)--(cr);
	\draw[->] (mur)--(cr);
	\draw[->] (cr)|-node[pos=0.4,below]{$r\left(t\right)$}node[very near end,below]{+}(sum);
	\draw[->] (wb)--(bf);
	\draw[->] (bf)--(sb);
	\draw[->] (sb)--(cb);
	\draw[->] (mub)--(cb);
	\draw[->] (cb)--node[right]{$\beta_n\left(t\right)$}(sat);
	\draw[->] (wd)--(df);
	\draw[->] (df)--(sd);
	\draw[->] (sd)--(cd);
	\draw[->] (mud)--(cd);
	\draw[->] (cd)--node[left]{$d\left(t\right)$}(dsum);
	\draw[->] (sum)--node[above]{$e\left(t\right)$}(C);
	\draw[->] (C)--node[above]{$u\left(t\right)$}(sat);
	\draw[->] (sat)--node[above]{$v\left(t\right)$}(dsum);
	\draw[->] (dsum)--node[above]{$z\left(t\right)$}(P);
	\end{tikzpicture}
	\caption{Block Diagram of Control System\label{fig:Block-Diagram-of}}
\end{figure*}
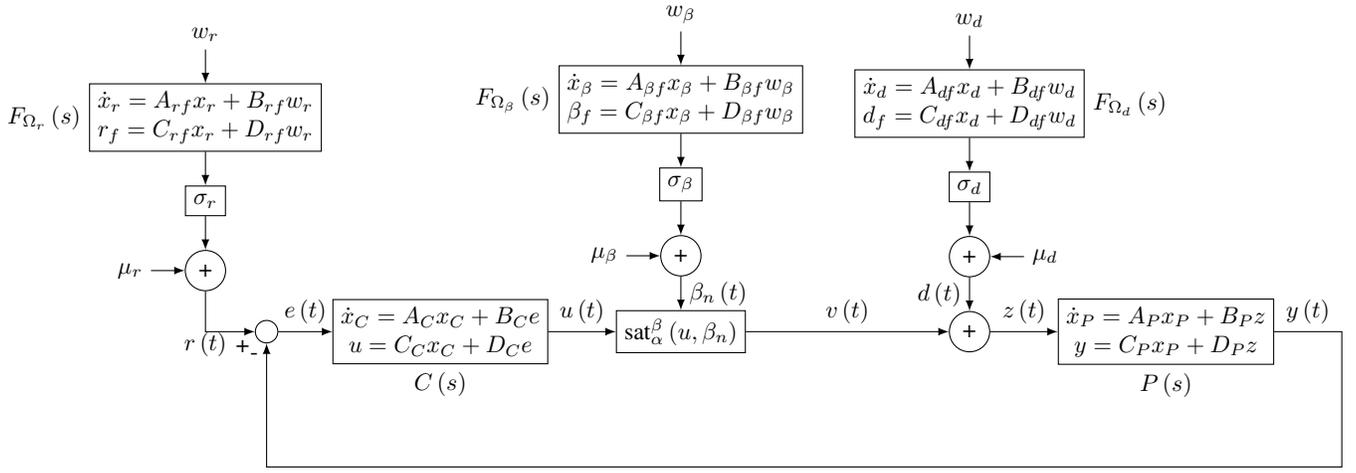

\begin{equation}\label{n1e1}
	N_{1}=\int_{\max\left(-\beta,\alpha\right)}^{\infty}\int_{\alpha-u_{2}}^{\beta+u_{2}}\left(1\right)\mathcal{N}\left(\mu_{1},\mu_{2},\sigma_{1},\sigma_{2},\rho\right)du_{1}du_{2}
	\end{equation}
	
	\begin{equation}\label{n2e1}
	\begin{gathered}
	N_{2}=\int_{\max\left(-\beta,\alpha\right)}^{\infty}\int_{-\infty}^{\alpha-u_{2}}\left(-1\right)\mathcal{N}\left(\mu_{1},\mu_{2},\sigma_{1},\sigma_{2},\rho\right)du_{1}du_{2}\\
	+\int_{\max\left(-\beta,\alpha\right)}^{\infty}\int_{\beta+u_{2}}^{\infty}\left(1\right)\mathcal{N}\left(\mu_{1},\mu_{2},\sigma_{1},\sigma_{2},\rho\right)du_{1}du_{2}
	\end{gathered}
	\end{equation}
	where $\mu_1$ and $\mu_2$ are the means of the inputs $u_1$ and $u_2$ respectively, $\sigma_1$ and $\sigma_2$ being their corresponding standard deviations, $\rho$ is the correlation between $u_1$ and $u_2$, and:
{\begin{equation}
\begin{gathered}\label{npdf}
\mathcal{N}\left(\mu_{1},\mu_{2},\sigma_{1},\sigma_{2},\rho\right)=\frac{1}{2\pi\sigma_{1}\sigma_{2}\sqrt{1-\rho^{2}}}e^{-\frac{{u}_{1}^{*2}+{u}_{2}^{*2}-2\rho u_{1}^{*}u_{2}^{*}}{2\left(1-\rho^{2}\right)}}
	\end{gathered}
\end{equation}}
is the bivariate Gaussian PDF in which
\[u_{1}^*=\frac{u_{1}-\mu_{1}}{\sigma_{1}}\]
\[u_{2}^*=\frac{u_{2}-\mu_{2}}{\sigma_{2}}\]
	
	The value of $M$ can be found from \eqref{meqd} and can be written as:
	\begin{equation}\label{me1}
	\small
	\begin{gathered}
	M=\int_{\max\left(-\beta,\alpha\right)}^{\infty}\int_{-\infty}^{\alpha-u_{2}}\left(\alpha-u_{2}\right)\mathcal{N}\left(\mu_{1},\mu_{2},\sigma_{1},\sigma_{2},\rho\right)du_{1}du_{2}\\
	+\int_{\max\left(-\beta,\alpha\right)}^{\infty}\int_{\alpha-u_{2}}^{\beta+u_{2}}u_{1}\mathcal{N}\left(\mu_{1},\mu_{2},\sigma_{1},\sigma_{2},\rho\right)du_{1}du_{2}\\
	+\int_{\max\left(-\beta,\alpha\right)}^{\infty}\int_{\beta+u_{2}}^{\infty}\left(\beta+u_{2}\right)\mathcal{N}\left(\mu_{1},\mu_{2},\sigma_{1},\sigma_{2},\rho\right)du_{1}du_{2}
	\end{gathered}
	\end{equation}

The integrals in the RHS of \eqref{n1e1}, \eqref{n2e1} and \eqref{me1} do not have closed form expressions, but can be computed numerically, for example, using vectorized adaptive quadrature, which is used by \texttt{integral2} in MATLAB \cite{shampine}. It has been practically observed that numerical computation is much slower using the integrals as they are. They are much faster to compute, by more than 2 orders of magnitude, if they are transformed as follows and resulting the inner integral reduced to a closed form expression. Applying the transformations:

\[u'_{1}=\frac{\frac{u_{1}-\mu_{1}}{\sigma_{1}}-\rho\frac{u_{2}-\mu_{2}}{\sigma_{2}}}{\sqrt{1-\rho^{2}}}\] and 
\[u'_{2}=\frac{u_{2}-\mu_{2}}{\sigma_{2}}\]
we get:
\[N_1=\int_{u'_{2min}}^{\infty}\int_{u'_{1min}}^{u'_{1max}}\left(1\right)\frac{1}{2\pi}e^{-\left(u_{1}^{'2}+u{}_{2}^{'2}\right)}d'u_{1}d'u_{2}\]
where:
\begin{equation}\label{up1min}
u'_{1\textrm{min}} = \frac{\alpha-\left(\mu_{1}+\mu_{2}\right)+u'_{2}\left(\rho\sigma_{1}+\sigma_{2}\right)}{\sigma_{1}\sqrt{1-\rho^{2}}}
\end{equation}
\begin{equation}\label{up1max}
u'_{1\textrm{max}}=\frac{\beta-\left(\mu_{1}-\mu_{2}\right)+u'_{2}\left(\sigma_{2}-\rho\sigma_{1}\right)}{\sigma_{1}\sqrt{1-\rho^{2}}}
\end{equation}
\begin{equation}\label{u2pmin}
u'_{2\textrm{min}}=\frac{\max\left(\alpha,-\beta\right)-\mu_{2}}{\sigma_{2}}
\end{equation}

On simplifying the inner integral, we get:
\begin{equation}\label{n1e2}
N_1=\int_{u'_{2\textrm{min}}}^{\infty}\frac{\sqrt{2}e^{-\frac{{u'_{2}}^{2}}{2}}}{4\sqrt{\pi}}\left[\mathrm{erf}\left(\gamma_1\right)+\mathrm{erf}\left(\gamma_2\right)\right]du'_{2}
\end{equation}
where:
\begin{equation}\label{g1}
\gamma_{1}=\frac{\sqrt{2}\left(\mu_{1}-\alpha+\mu_{2}+\sigma_{2}u'_{2}+\rho\sigma_{1}u'_{2}\right)}{2\sigma_{1}\sqrt{1-\rho^{2}}}
\end{equation}
\begin{equation}\label{g2}
\gamma_{2}=\frac{\sqrt{2}\,\left(\mathrm{\beta}-\mu_{1}+\mu_{2}+\sigma_{2}\,u'_{2}-\rho\,\sigma_{1}\,u'_{2}\right)}{2\sigma_{1}\sqrt{1-\rho^{2}}}
\end{equation}

Similarly,
\begin{equation*}
\begin{gathered}
N_{2}=\int_{u'_{2\textrm{min}}}^{\infty}\int_{-\infty}^{u'_{1\textrm{min}}}\left(-1\right)\frac{1}{2\pi}e^{-\left(u_{1}^{'2}+u{}_{2}^{'2}\right)}du'_{1}du'_{2}
\\+\int_{u'_{2\textrm{min}}}^{\infty}\int_{u'_{1\textrm{max}}}^{\infty}\left(1\right)\frac{1}{2\pi}e^{-\left(u_{1}^{'2}+u{}_{2}^{'2}\right)}du'_{1}du'_{2}
\end{gathered}
\end{equation*}
On simplifying,
\begin{equation}\label{n2e2}
N_{2}=\int_{u'_{2\textrm{min}}}^{\infty}\frac{\sqrt{2}e^{-\frac{u_{2}^{2}}{2}}}{4\,\sqrt{\pi}}\left[\mathrm{erf}\left(\gamma_{1}\right)-\mathrm{erf}\left(\gamma_{2}\right)\right]du'_{2}
\end{equation}

Finally,
\begin{equation*}
    \begin{gathered}
	M=\int_{u'_{2\textrm{min}}}^{\infty}\left[\int_{-\infty}^{u'_{1\textrm{min}}}\left(\alpha-\mu_{2}-u'_{2}\sigma_{2}\right)\frac{1}{2\pi}e^{-\left(u_{1}^{'2}+u{}_{2}^{'2}\right)}du'_{1}\right.\\+\int_{u'_{1\textrm{min}}}^{u'_{1\textrm{max}}}\left(\mu_{1}+\sigma_{1}\sqrt{1-\rho^{2}}u'_{1}+\rho\sigma_{1}u'_{2}\right)\frac{1}{2\pi}e^{-\left(u_{1}^{'2}+u{}_{2}^{'2}\right)}du'_{1}\\\left.+\int_{u'_{1\textrm{max}}}^{\infty}\left(\beta+\mu_{2}+u'_{2}\sigma_{2}\right)\frac{1}{2\pi}e^{-\left(u_{1}^{'2}+u{}_{2}^{'2}\right)}du'_{1}\right]du'_{2}
    \end{gathered}
\end{equation*}
On simplifying,
\begin{equation}\label{me2}
\begin{gathered}
M=\int_{u'_{2 \textrm{min}}}^{\infty}\frac{\sqrt{2}\,{\mathrm{e}}^{-\frac{{u'_{2}}^{2}}{2}}\,\left(\mathrm{erf}\left(\gamma_{1}\right)-1\right)\,\left(\mu_{2}-\alpha+\sigma_{2}\,u'_{2}\right)}{4\,\sqrt{\pi}}du'_{2}\\
+\int_{u'_{2 \textrm{min}}}^{\infty}\frac{\sigma_{1}\sqrt{1-\rho^{2}}}{2\pi}e^{-\frac{{u'}_{2}^{2}}{2}}\left(e^{-\gamma_{1}^{2}}-e^{-\gamma_{2}^{2}}\right)du'_{2}\\
+\int_{u'_{2\textrm{min}}}^{\infty}\frac{\sqrt{2}}{4\sqrt{\pi}}\left(\mu_{1}+\rho\sigma_{1}u'_{2}\right){\mathrm{e}}^{-\frac{{u'_{2}}^{2}}{2}}\left[\mathrm{erf}\left(\gamma_{1}\right)+\mathrm{erf}\left(\gamma_{2}\right)\right]du'_{2}\\
-\int_{u'_{2 \textrm{min}}}^{\infty}\frac{\sqrt{2}\,{\mathrm{e}}^{-\frac{{u'_{2}}^{2}}{2}}\,\left(\mathrm{erf}\left(\gamma_{2}\right)-1\right)\,\left(\beta+\mu_{2}+\sigma_{2}\,u'_{2}\right)}{4\,\sqrt{\pi}}du'_{2}
\end{gathered}
\end{equation}

\section{Quasilinear Control of Closed Loop Control System}\label{qlccl}
\subsection{Description of the Closed Loop System}
Consider the control system shown in Fig. \ref{fig:Block-Diagram-of}. Similar to the system in Fig. \ref{figsv}, it has a plant $P\left(s\right)$ and a controller $C\left(s\right)$. However, this time, the actuator is a bivariate saturation $\textrm{sat}_{\alpha,\beta}\left[u_1\left(t\right),u_2\left(t\right)\right]$, with random bounds modeled as a second input $u_2\left(t\right)$. Similar to the generation of the reference $r\left(t\right)$ and the disturbance $d\left(t\right)$ described in subsection \ref{clssv}, the second actuator input $u_2(t)$ is generated by passing white noise $w_{\beta}(t)$ through a coloring filter $F_{\Omega_{\beta}}(s)$ with $H_2$-norm equal to 1, scaling it by $\sigma_2$, and adding a bias $\mu_2$. This ensures that he reference $r\left(t\right)$, the disturbance $d\left(t\right)$ and the second actuator input $u_2(t)$ are WSS Gaussian random processes with means $\mu_{r}$, $\mu_{d}$, $\mu_2$, and standard deviations $\sigma_{r},\sigma_{d}$ and $\sigma_2$ respectively. The coloring filters band-limit the white noises $w_r$, $w_d$ and $w_{\beta}$ to a desired bandwidth, which is desirable to be close to the system bandwidth.  The block diagram shows the corresponding state-space representations.

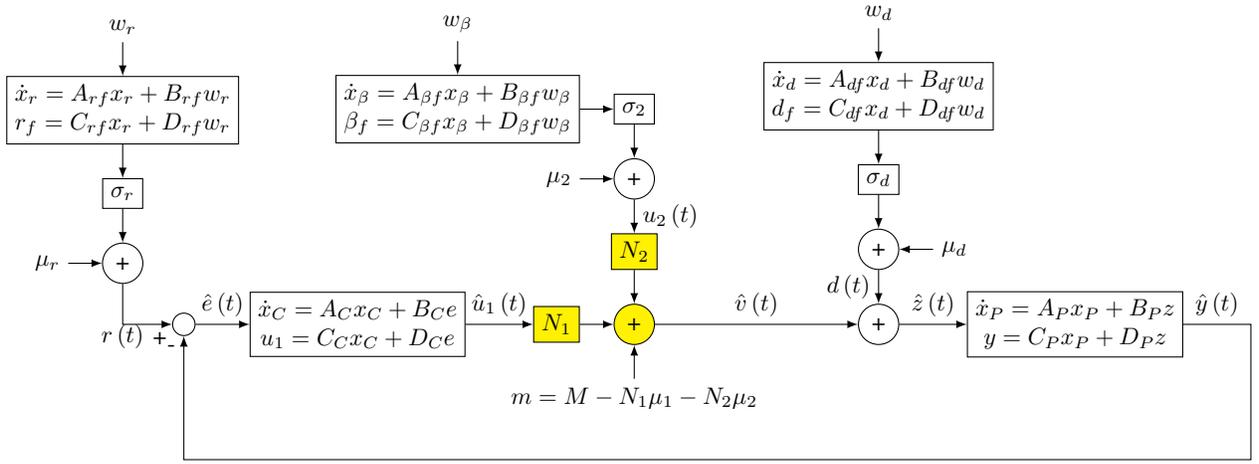
\begin{figure*}[t]
	\centering
	\begin{tikzpicture}[>=latex,scale=0.9,every node/.style={transform shape}, every text node part/.style={align=center}]
	\node (wr) at (0,0){$w_r$};
	\node[rectangle,draw] (rf) [below=0.5 of wr]{$\dot{x}_r=A_{rf}x_r+B_{rf}w_r$\\$r_f=C_{rf}x_r+D_{rf}w_r$};
	\node[rectangle,draw] (sr) [below=0.5 of rf]{$\sigma_r$};
	\node[circle,draw] (cr) [below=0.5 of sr]{+};
	\node (mur) [left=0.5 of cr]{$\mu_r$};
	\node[circle,draw] (sum) [below right=0.8 of cr]{};
	\node[rectangle,draw] (C) [right=0.8 of sum]{$\dot{x}_C=A_{C}x_C+B_{C}e$\\$u_1=C_{C}x_C+D_Ce$};
	\node[rectangle,draw,fill=yellow] (N1) [right= of C]{$N_1$};
	\node[circle,draw,fill=yellow] (satc)[right=0.5 of N1]{+};
	\node[circle,draw] (dsum) [right=3 of satc]{+};
	
	\node[rectangle,draw,fill=yellow] (N2) [above=0.5 of satc]{$N_2$};
	\node (m) [below=0.5 of satc]{$m=M-N_1\mu_1-N_2\mu_{2}$};
	\node[circle,draw] (cb) [above=0.5 of N2]{+};
	\node (mub) [left=0.5 of cb]{$\mu_{2}$};
	\node[rectangle,draw] (sb) [above=0.5 of cb]{$\sigma_{2}$};
	\node[rectangle,draw] (bf) [left=0.5 of sb]{$\dot{x}_{\beta}=A_{{\beta}f}x_{\beta}+B_{{\beta}f}w_{\beta}$\\${\beta}_f=C_{\beta f}x_{\beta}+D_{\beta f}w_{\beta}$};
	\node (wb) [above=0.5 of bf]{$w_{\beta}$};
	
	\node[circle,draw] (cd) [above=0.5 of dsum]{+};
	\node (mud) [right=0.5 of cd]{$\mu_{d}$};
	\node[rectangle,draw] (sd) [above=0.5 of cd]{$\sigma_{d}$};
	\node[rectangle,draw] (df) [above=0.5 of sd]{$\dot{x}_{d}=A_{df}x_{d}+B_{df}w_{d}$\\${d}_f=C_{df}x_{d}+D_{df}w_{d}$};
	\node (wd) [above=0.5 of df]{$w_{d}$};
	
	\node[rectangle,draw] (P) [right= of dsum]{$\dot{x}_P=A_{P}x_P+B_{P}z$\\$y=C_{P}x_P+D_Pz$};
	\node[coordinate] (y) [right= of P]{};
	
	\draw[->] (P)--node[above]{$\hat{y}\left(t\right)$}(y)--($(y)+(0,-2)$)-|node[pos=0.95,left]{-}(sum);
	\draw[->] (wr)--(rf);
	\draw[->] (rf)--(sr);
	\draw[->] (sr)--(cr);
	\draw[->] (mur)--(cr);
	\draw[->] (cr)|-node[pos=0.4,below]{$r\left(t\right)$}node[very near end,below]{+}(sum);
	\draw[->] (wb)--(bf);
	\draw[->] (bf)--(sb);
	\draw[->] (sb)--(cb);
	\draw[->] (mub)--(cb);
	\draw[->] (cb)--node[right]{$u_2\left(t\right)$}(N2);
	\draw[->] (wd)--(df);
	\draw[->] (df)--(sd);
	\draw[->] (sd)--(cd);
	\draw[->] (mud)--(cd);
	\draw[->] (cd)--node[left]{$d\left(t\right)$}(dsum);
	\draw[->] (sum)--node[above]{$\hat{e}\left(t\right)$}(C);
	\draw[->] (C)--node[above]{$\hat{u}_1\left(t\right)$}(N1);
	\draw[->] (N1)--(satc);
	\draw[->] (N2)--(satc);
	\draw[->] (m)--(satc);
	\draw[->] (satc)--node[above]{$\hat{v}\left(t\right)$}(dsum);
	\draw[->] (dsum)--node[above]{$\hat{z}\left(t\right)$}(P);
	\end{tikzpicture}
	\caption{Stochastically linearized version of Fig. \ref{fig:Block-Diagram-of}}
	\label{sl}
\end{figure*}

Applying stochastic linearization to this system, we get the system of Fig. \ref{sl}. To find the values of the quasilinear gains $N_1$, $N_2$ and the quasilinear bias $M$, it is required to compute the means and standard deviations of $u_1(t)$ and $u_2(t)$. This process can be simplified by separating the input signals into their zero-mean random parts and a constant mean part, considering two different systems and adding the results, as investigated in the following subsection.

\subsection{Decomposition into two sub-systems}

Consider the following stochastic state space equation as a general representation of the system in Fig.~\ref{sl}:
\begin{equation}\label{sde}
d\mathbf{x}=\mathbf{Ax}dt+\sum_{i=1}^{n}\mathbf{b_{i}}dw_{i}
\end{equation}
where $\mathbf{x}=\left[\begin{array}{cccc}
x_{1} & x_{2} & \ldots & x_{n}\end{array}\right]^T$ are the states of the system and $w_i$ are the stochastic inputs. It is known that, provided the system is asymptotically stable, the covariance matrix associated with $\mathbf{x}$, $\mathbf{\Sigma}$, is the solution of the following Lyapunov equation \cite{Brockett}:
\begin{equation}\label{lyapeq}
\mathbf{A}\mathbf{\Sigma}+\mathbf{\Sigma}\mathbf{A^T}+\sum_{i=1}^{n}\mathbf{b_{i}}\mathbf{b_{i}}^{\mathbf{T}}=\mathbf{0}
\end{equation}
where:
{\footnotesize\begin{align}\label{sigma}
\Sigma & =\left[\begin{array}{cccc}
E\left[x_{1}^{2}\left(t\right)\right] & E\left[x_{1}\left(t\right)x_{2}\left(t\right)\right] & \cdots & E\left[x_{1}\left(t\right)x_{n}\left(t\right)\right]\\
E\left[x_{2}\left(t\right)x_{1}\left(t\right)\right] & E\left[x_{2}^{2}\left(t\right)\right] & \ddots & E\left[x_{2}\left(t\right)x_{n}\left(t\right)\right]\\
\vdots & \ddots & \ddots & \vdots\\
E\left[x_{n}\left(t\right)x_{1}\left(t\right)\right] & E\left[x_{n}\left(t\right)x_{2}\left(t\right)\right] & \cdots & E\left[x_{n}^{2}\left(t\right)\right]
\end{array}\right]
\end{align}}is the covariance matrix. For the following Theorem, we denote a system described by \eqref{sde} by $S_{\mathbf{\Sigma}}$ where $\mathbf{\Sigma}$ is its corresponding covariance matrix from \eqref{lyapeq}.

\begin{theorem}
Consider a system $S_{\mathbf{\Sigma}}$, having a covariance matrix $\mathbf{\Sigma}$. Also consider two other systems, $S_{\mathbf{\Sigma}_1}$ and $S_{\mathbf{\Sigma}_2}$, having covariance matrices $\mathbf{\Sigma}_1$ and $\mathbf{\Sigma}_2$ respectively. One of them, $S_{\mathbf{\Sigma}_1}$, is driven by the zero-mean parts of the inputs to system $S_{\mathbf{\Sigma}}$. The other system, $S_{\mathbf{\Sigma}_2}$, is driven by constant inputs having values equal to the means of the inputs driving $S_{\mathbf{\Sigma}}$. Then $\mathbf{\Sigma}=\mathbf{\Sigma}_1+\mathbf{\Sigma}_2$.
\end{theorem}
\begin{proof}
The correlation coefficient between any two states is defined by:
\[
\rho_{ij}=\frac{\textrm{cov}\left(x_{i},x_{j}\right)}{\sigma_{i}\sigma_{j}}=\frac{E\left[x_{i}x_{j}\right]-\mu_{i}\mu_{j}}{\sigma_{i}\sigma_{j}}
\]
where $\mu_i=E\left[x_i\right]$ and $\sigma_{i}=\sqrt{E\left(x_{i}^{2}\right)-\left\{ E\left(x_{i}\right)\right\} ^{2}}$. Then,
\[
E\left[x_{i}x_{j}\right]=\rho_{ij}\sigma_{i}\sigma_{j}+\mu_{i}\mu_{j}
\]
and hence from \eqref{sigma}:
{\scriptsize 
\[
\mathbf{\Sigma} =\left[\begin{array}{cccc}
\sigma_{1}^{2}+\mu_{1}^{2} & \rho_{12}\sigma_{1}\sigma_{2}+\mu_{1}\mu_{2} & \cdots & \rho_{1n}\sigma_{1}\sigma_{n}+\mu_{1}\mu_{n}\\
\rho_{21}\sigma_{2}\sigma_{1}+\mu_{2}\mu_{1} & \sigma_{2}^{2}+\mu_{2}^{2} & \ddots & \rho_{2n}\sigma_{2}\sigma_{n}+\mu_{2}\mu_{n}\\
\vdots & \ddots & \ddots & \vdots\\
\rho_{n1}\sigma_{n}\sigma_{1}+\mu_{n}\mu_{1} & \rho_{n2}\sigma_{n}\sigma_{2}+\mu_{n}\mu_{2} & \cdots & \sigma_{n}^{2}+\mu_{n}^{2}
\end{array}\right]\]}
\[
=\left[\begin{array}{cccc}
\sigma_{1}^{2} & \rho_{12}\sigma_{1}\sigma_{2} & \cdots & \rho_{1n}\sigma_{1}\sigma_{n}\\
\rho_{21}\sigma_{2}\sigma_{1} & \sigma_{2}^{2} & \ddots & \rho_{2n}\sigma_{2}\sigma_{n}\\
\vdots & \ddots & \ddots & \vdots\\
\rho_{n1}\sigma_{n}\sigma_{1} & \rho_{n2}\sigma_{n}\sigma_{2} & \cdots & \sigma_{n}^{2}
\end{array}\right]\]
\begin{equation}\label{sigma1}
+\left[\begin{array}{cccc}
\mu_{1}^{2} & \mu_{1}\mu_{2} & \cdots & \mu_{1}\mu_{n}\\
\mu_{2}\mu_{1} & \mu_{2}^{2} & \ddots & \mu_{2}\mu_{n}\\
\vdots & \ddots & \ddots & \vdots\\
\mu_{n}\mu_{1} & \mu_{n}\mu_{2} & \cdots & \mu_{n}^{2}
\end{array}\right]
\end{equation}
When the inputs are zero-mean, so are all the states, and hence, $\mu_i=0$ for $i=1,2,...,n$. Then $\mathbf{\Sigma}=\mathbf{\Sigma}_1$. When the inputs are constant and are equal to the means of the inputs to $S_{\mathbf{\Sigma}}$, there is also no variability in the states, and hence, $\sigma_i=0$ for $i=1,2,...,n$ and $\mathbf{\Sigma}=\mathbf{\Sigma}_2$. Therefore, in general, from \eqref{sigma1},
\[\mathbf{\Sigma}=\mathbf{\Sigma}_1+\mathbf{\Sigma}_2\]
\end{proof}
Hence, for the purpose of analysis, the system can be decomposed into two systems - one having all inputs as zero-mean (this is useful for computing the standard deviations of required signals) and another one - with constant inputs representing the averages of the inputs (this is useful for computing the expected values of the required signals). 

\subsection{Calculation of mean and standard deviation of actuator inputs}

\begin{figure*}[t]
	\begin{equation}\label{e1}
	\mathbf{A}=\left(\begin{array}{ccccc}
	A_{rf} & 0 & 0 & 0 & 0\\
	0 & A_{\beta f} & 0 & 0 & 0\\
	0 & 0 & A_{df} & 0 & 0\\
	\frac{B_{c}\,C_{rf}\,\sigma_{r}}{D_{c}D_{p}N_{1}+1} & -\frac{B_{c}\,C_{\beta f}\,D_{p}\,N_{2}\,\sigma_{\beta}}{D_{c}D_{p}N_{1}+1} & -\frac{B_{c}\,C_{df}\,D_{p}\,\sigma_{d}}{D_{c}D_{p}N_{1}+1} & \frac{A_{c}+A_{c}\,D_{c}\,D_{p}\,N_{1}-B_{c}\,C_{c}\,D_{p}\,N_{1}}{D_{c}D_{p}N_{1}+1} & -\frac{B_{c}\,C_{p}}{D_{c}D_{p}N_{1}+1}\\
	\frac{B_{p}\,C_{rf}\,D_{c}\,N_{1}\,\sigma_{r}}{D_{c}D_{p}N_{1}+1} & \frac{B_{p}\,C_{\beta f}\,N_{2}\,\sigma_{\beta}}{D_{c}D_{p}N_{1}+1} & \frac{B_{p}\,C_{df}\,\sigma_{d}}{D_{c}D_{p}N_{1}+1} & \frac{B_{p}\,C_{c}\,N_{1}}{D_{c}D_{p}N_{1}+1} & \frac{A_{p}+A_{p}\,D_{c}\,D_{p}\,N_{1}-B_{p}\,C_{p}\,D_{c}\,N_{1}}{D_{c}D_{p}N_{1}+1}\\
	\frac{C_{rf}\,D_{c}\,\sigma_{r}}{D_{c}D_{p}N_{1}+1} & -\frac{C_{\beta f}\,D_{c}\,D_{p}\,N_{2}\,\sigma_{\beta}}{D_{c}D_{p}N_{1}+1} & -\frac{C_{df}\,D_{c}\,D_{p}\,\sigma_{d}}{D_{c}D_{p}N_{1}+1} & \frac{C_{c}}{D_{c}D_{p}N_{1}+1} & -\frac{C_{p}\,D_{c}}{D_{c}D_{p}N_{1}+1}\\
	0 & C_{\beta f}\,\sigma_{\beta} & 0 & 0 & 0
	\end{array}\right)
	\end{equation}
	\begin{equation}\label{e2}
	\mathbf{B}=\left(\begin{array}{ccc}
	B_{rf} & 0 & 0\\
	0 & 0 & B_{bf}\\
	0 & B_{df} & 0\\
	\frac{B_{c}\,D_{rf}\,\sigma_{r}}{D_{c}D_{p}N_{1}+1} & -\frac{B_{c}\,D_{df}\,D_{p}\,\sigma_{d}}{D_{c}D_{p}N_{1}+1} & -\frac{B_{c}\,D_{bf}\,D_{p}\,N_{2}\,\sigma_{\beta}}{D_{c}D_{p}N_{1}+1}\\
	\frac{B_{p}\,D_{c}\,D_{rf}\,N_{1}\,\sigma_{r}}{D_{c}D_{p}N_{1}+1} & \frac{B_{p}\,D_{df}\,\sigma_{d}}{D_{c}D_{p}N_{1}+1} & \frac{B_{p}\,D_{bf}\,N_{2}\,\sigma_{\beta}}{D_{c}D_{p}N_{1}+1}\\
	\frac{D_{c}\,D_{rf}\,\sigma_{r}}{D_{c}D_{p}N_{1}+1} & -\frac{D_{c}\,D_{df}\,D_{p}\,\sigma_{d}}{D_{c}D_{p}N_{1}+1} & -\frac{D_{bf}\,D_{c}\,D_{p}\,N_{2}\,\sigma_{\beta}}{D_{c}D_{p}N_{1}+1}\\
	0 & 0 & D_{bf}\,\sigma_{\beta}
	\end{array}\right)
	\end{equation} 
\end{figure*}

The values of $\mathbf{A}$ and $\mathbf{B}$ for the system of Fig. \ref{sl} are shown in \eqref{e1} and \eqref{e2}. The corresponding state vector is:
\[\mathbf{x}(t)=\left[\begin{array}{ccccc}
x_{rf} & x_{\beta f} & x_{df} & x_{C} & x_{P}\end{array}\right]^T\]
with the states as depicted in Fig. \ref{sl}. To calculate the values of $N_1$, $N_2$ and $M$, the values of mean and standard deviation of the first actuator input, $\hat{u}_1\left(t\right)$, are required, along with correlation between the two actuator inputs. This is obtained by solving the Lyapunov equation \eqref{lyapeq} with $\sum_{i=1}^{n}\mathbf{b_{i}}\mathbf{b_{i}}^{\mathbf{T}}=~\mathbf{BB^T}$. The following formula can be used for this purpose \cite{Jarlebring}:
\begin{equation}\label{lyap}
\left(\mathbf{I_{n}}\otimes\mathbf{A}+\mathbf{A^{T}}\otimes\mathbf{I_{n}}\right)\textrm{vec}\left(\mathbf{\Sigma}\right)=-\textrm{vec}\left(\mathbf{B}\right)
\end{equation}
where $\mathbf{\Sigma}=E\left[\mathbf{x}(t)\mathbf{x^T}(t)\right]$, $\mathbf{I_{n}}$ is the $n\times n$ identity matrix, and $\textrm{vec}\left(\cdot\right)$ is the vectorization operator.
First, consider only the zero-mean parts of the signals. In that case, if $\hat{u}_1(t)=\mathbf{C_{1}}\mathbf{x}(t)$, then:
\begin{equation}\label{Csu}
\mathbf{C_{1}}=\left[\begin{array}{c}
\frac{C_{rf}D_{c}\sigma_{r}}{D_{c}D_{p}N_{1}+1}\\
\frac{C_{bf}D_{c}D_{p}N_{2}\sigma_{\beta}}{D_{c}D_{p}N_{1}+1}\\
\frac{C_{df}D_{c}D_{p}\sigma_{d}}{D_{c}D_{p}N_{1}+1}\\
\frac{C_{c}}{D_{c}D_{p}N_{1}+1}\\
\frac{C_{p}D_{c}}{D_{c}D_{p}N_{1}+1}
\end{array}\right]^{T}
\end{equation}
Hence,
\begin{equation}\label{s1}
\sigma_{\hat{1}}^2=\mathbf{C_{1}}\mathbf{\Sigma}\mathbf{C_{1}}^T
\end{equation}
where $\sigma_{\hat{1}}$ is the standard deviation of $\hat{u}_1$ and $\mathbf{\Sigma}$ is as defined in \eqref{sigma}.

Taking only the means of the signals as constant inputs to the system, the mean of the actuator input can be found to be:
\begin{equation}\label{muu}
\mu_{\hat{1}}=-\frac{\left(m+\mu_{d}+N_{2}\mu_{2}-\frac{1}{P_{dc}\mu_{r}}\right)}{N_{1}+\frac{1}{C_{dc}P_{dc}}}
\end{equation}
where $C_{dc}$ and $P_{dc}$ are the DC gains of $C(s)$ and $P(s)$ respectively, and
\begin{equation}\label{meq}
m=M-N_1\mu_{\hat{1}}-N_2\mu_{2}
\end{equation}

\subsection{Correlation between Actuator Inputs}
To find the correlation coefficient between the actuator inputs $\hat{u}_1(t)$ and $u_2(t)$, it is noted that if $u_2(t)=\mathbf{C_{2}x}(t)$, then:
\begin{equation}\label{Cbeta}
\mathbf{C_{2}}=\left[\begin{array}{c}
	0\\
	C_{bf}\sigma_{2}\\
	0\\
	0\\
	0
\end{array}\right]^{T}\end{equation}
Hence, the correlation coefficient:
\begin{equation}\label{rho}
\rho=\frac{\mathbf{C_{1}\Sigma C_{2}^T}}{\sigma_{\hat{1}}\sigma_{2}}
\end{equation}

\subsection{Solution of Equations}
The values of $N_1$, $N_2$ and $M$ can be found by solving the system of equations \eqref{u2pmin}-\eqref{me2}, \eqref{e1}-\eqref{rho}. From \eqref{n1e2}, \eqref{n2e2} and \eqref{me2}, it can be seen that $N_1$, $N_2$ and $M$ are functions of $\mu_{\hat{1}}$, $\sigma_{\hat{1}}$, and $\rho$, i.e.,
\begin{equation}\label{n1nm}
N_{1}=\mathcal{F}_{N_{1}}\left(\mu_{\hat{1}},\sigma_{\hat{1}},\rho\right)
\end{equation}
\begin{equation}\label{n2nm}
N_{2}=\mathcal{F}_{N_{2}}\left(\mu_{\hat{1}},\sigma_{\hat{1}},\rho\right)
\end{equation}
\begin{equation}\label{mnm}
M=\mathcal{F}_{M}\left(\mu_{\hat{1}},\sigma_{\hat{1}},\rho\right)
\end{equation}
Using \eqref{muu}, \eqref{meq} and \eqref{mnm}, the following equation can be written,
\begin{equation}\label{meq1}
\frac{1}{P_{dc}\mu_{r}}-\frac{\mu_{\hat{1}}}{C_{dc}P_{dc}}-\mu_{d}=\mathcal{F}_{M}\left(\mu_{\hat{1}},\sigma_{\hat{1}},\rho\right)
\end{equation}

Equation \eqref{muu} implies that $\mu_{\hat{1}}$ is a function of $N_1$, $N_2$ and $M$, i.e.,
\begin{equation}\label{mu1nm}
\mu_{\hat{1}}=\mathcal{F}_{\mu_{\hat{1}}}\left(N_{1},N_{2},M\right)
\end{equation}
Furthermore, \eqref{s1} and \eqref{rho} imply that $\sigma_{\hat{1}}$ and $\rho$ are functions of $N_1$ and $N_2$, i.e.,
\begin{equation}\label{s1nm}
\sigma_1=\mathcal{F}_{\sigma_{\hat{1}}}\left(N_{1},N_{2}\right)
\end{equation}
\begin{equation}\label{rnm}
\rho=\mathcal{F}_{\rho}\left(N_{1},N_{2}\right)
\end{equation}
Using \eqref{mu1nm}-\eqref{rnm}, the system of equations \eqref{u2pmin}-\eqref{me2}, \eqref{e1}-\eqref{rho} can be reduced to an equivalent system consisting of \eqref{n1nm}, \eqref{n2nm} and \eqref{meq1}, which together with \eqref{mu1nm}-\eqref{rnm}, can be written as a system of 3 equations with 3 unknowns, $N_1$, $N_2$ and $M$:
\begin{equation*}
N_{1}=\mathcal{F}_{N_{1}}\left(\mathcal{F}_{\mu_{\hat{1}}}\left(N_{1},N_{2},M\right),\mathcal{F}_{\sigma_{\hat{1}}}\left(N_{1},N_{2}\right),\mathcal{F}_{\rho}\left(N_{1},N_{2}\right)\right)
\end{equation*}
\begin{equation*}
N_{2}=\mathcal{F}_{N_{2}}\left(\mathcal{F}_{\mu_{\hat{1}}}\left(N_{1},N_{2},M\right),\mathcal{F}_{\sigma_{\hat{1}}}\left(N_{1},N_{2}\right),\mathcal{F}_{\rho}\left(N_{1},N_{2}\right)\right)
\end{equation*}
\begin{equation*}
\begin{gathered}
\frac{1}{P_{dc}\mu_{r}}-\frac{\mathcal{F}_{\mu_{\hat{1}}}\left(N_{1},N_{2},M\right)}{C_{dc}P_{dc}}-\mu_{d}\\
=\mathcal{F}_{M}\left(\mathcal{F}_{\mu_{\hat{1}}}\left(N_{1},N_{2},M\right),\mathcal{F}_{\sigma_{\hat{1}}}\left(N_{1},N_{2}\right),\mathcal{F}_{\rho}\left(N_{1},N_{2}\right)\right)
\end{gathered}
\end{equation*}

A sufficient condition of the existence of solutions for the system of equations \eqref{n1nm}, \eqref{n2nm} and \eqref{meq1} is discussed in the following Theorem.

\begin{theorem}\label{exist}
Let $\mathcal{N}_{1}$, $\mathcal{N}_{2}$ and $\mathcal{M}$ denote the ranges of $\mathcal{F}_{N_{1}}$, $\mathcal{F}_{N_{2}}$ and $\mathcal{F}_{M}$ respectively in \eqref{n1nm}-\eqref{mnm}, and let $\overline{\mathcal{N}}_{1}$, $\overline{\mathcal{N}}_{2}$ and $\overline{\mathcal{M}}$ denote their closures. 
Assume that the following hold:

\begin{enumerate}
\item All the eigenvalues of $\mathbf{A}$ in \eqref{e1} are in the open left half plane $\forall N_1\in{\overline{\mathcal{N}}_{1}}$, $N_2\in{\overline{\mathcal{N}}_{2}}$.
\item $\overline{\mathcal{N}}_{1}$, $\overline{\mathcal{N}}_{2}$ and $\mathcal{\overline{M}}$ are compact and convex sets.
\item If $C_{dc}=\infty$, then $\frac{1}{P_{dc}\mu_{r}}-\mu_{d}\in\mathcal{M}$. If $P_{dc}=\infty$ but $C_{dc}\neq\infty$, then $-\mu_{d}\in\mathcal{M}$.
\item If $C_{dc}=\infty$ or $P_{dc}=\infty$ or both, then:
\[\left|\frac{\partial}{\partial\mu_1}\left[\mathcal{F}_M\left(\mu_{\hat{1}},\sigma_{\hat{1}},\rho\right)\right]\right|\geq d\]
for a fixed constant $d>0$.
\end{enumerate}

Then, the system of equations \eqref{n1nm}, \eqref{n2nm} and \eqref{meq1} has a solution in $\overline{\mathcal{N}}_{1}$, $\overline{\mathcal{N}}_{2}$ and $\overline{\mathcal{M}}$. 
\end{theorem}
\begin{proof}
We consider two cases. First assume that $C_{dc}\neq\infty$ and $P_{dc}\neq\infty$. The first assumption implies that for any value of $N_1\in{\overline{\mathcal{N}}_{1}}$ and $N_2\in{\overline{\mathcal{N}}_{2}}$, there is a unique positive definite solution of \eqref{lyap}. Hence, $\sigma_{\hat{1}}$ and $\rho$ exist, from \eqref{s1} and \eqref{rho} respectively, and are continuous functions of $N_1$ and $N_2$. Also, from \eqref{muu} and \eqref{meq}, $\mu_{\hat{1}}$ is a continuous function of $N_1$, $N_2$ and $M$. Therefore, the both the sides of \eqref{n1nm}, \eqref{n2nm} and \eqref{meq1} form continuous functions of $N_1$, $N_2$ and $M$.  Since the second assumption holds, by Brouwer's fixed point theorem \cite{smart1980}, \eqref{n1nm}, \eqref{n2nm} and \eqref{meq1} have a solution and the result follows.

For the second case, assume either $C_{dc}=\infty$ or $P_{dc}=~\infty$ or both. Then the LHS of \eqref{meq1} reduces to a constant, independent of $N_1$, $N_2$ and $M$. Since the range of $\mathcal{F}_{M}$ is $\mathcal{M}$, the third assumption ensures a necessary condition for \eqref{meq1} to have a solution. Let $\mathbf{p_1}=\left[\begin{array}{cc}
\sigma_{\hat{1}} & \rho\end{array}\right]^{T}$ and $\mathbf{p_2}=\mu_{\hat{1}}$. Since $\mu_{\hat{1}}$, $\sigma_{\hat{1}}$ and $\rho$ are continuous functions of $N_1$, $N_2$ and/or $M$, so are $\mathbf{p_1}$ and $\mathbf{p_2}$. In addition, since $\mathcal{F}_M$ is continuous and also continuously differentiable with respect to $\mu_{\hat{1}}$, $f\left(\mathbf{p_{1}},\mathbf{p_{2}}\right)=\mathcal{F}_M\left(\mu_{\hat{1}},\sigma_{\hat{1}},\rho\right)$ is a continuous mapping from $\mathbb{R}^{2}\rightarrow\mathbb{R}$ and is continuously differentiable in $\mathbf{x_2}$. By the fifth assumption and Theorem 1 in \cite{zhang2006}, for any value of $N_1\in{\overline{\mathcal{N}}_{1}}$ and $N_2\in{\overline{\mathcal{N}}_{2}}$, and hence for any $\sigma_{\hat{1}}$ and $\rho$ (which are continuous functions of $N_1$ and $N_2$), there exists a unique $\mu_{\hat{1}}=\mathbf{x_{2}}=g\left(\mathbf{x_{1}}\right)=h\left(\sigma_{\hat{1}},\rho\right)$ where $g$ and $h$ are continuous. Hence, the resulting RHS of \eqref{meq1} is continuous. By assumption 2 and Brouwer's fixed point theorem, the result follows.
\end{proof}

There is no closed form solution for this system of equations, but it can be numerically arrived at by using an algorithm like Trust-Region Dogleg \cite{powell1968}, which is used in, for example, MATLAB's \texttt{fsolve}.

\section{Comparison with single variable QLC}\label{cSQL}

The essential difference between single variable QLC and multi-variable QLC is the presence of the second actuator input $u_2(t)$. This leads to effects peculiar only to multi-variable QLC. Two of them are described in the following subsections.

\subsection{Effect of secondary input on primary input saturation}
A particular point of interest is the fact that whether the addition of noise to the actuator bounds leads to an increased saturation in the primary actuator input or not. The following Theorem investigates that.

\begin{theorem}\label{tp}
The probability that the primary input $U_1$ is not saturated in a bivariate saturation nonlinearity with jointly Gaussian inputs $U_1$ and $U_2$ is quantified by the input quasilinear gain, $N_1$, i.e.,
\[P\left(\alpha-U_{2}<U_{1}<\beta+U_{2}\right)=N_{1}\]
\end{theorem}
\begin{proof}
By definition of the joint probability distribution function,
\[
\begin{gathered}
P\left(\alpha-U_{2}<U_{1}<\beta+U_{2}\right)\\
=\int_{-\infty}^{\infty}\int_{\alpha-u_{2}}^{\beta+u_{2}}p_{U_{1},U_{2}}\left(u_{1},u_{2}\right)du_{1}du_{2}
\end{gathered}\]
where $p_{U_{1},U_{2}}\left(u_{1},u_{2}\right)$ is the joint probability distribution function (PDF) of $U_1$ and $U_2$. Since $U_1$ and $U_2$ are jointly Gaussian,
\[p_{U_{1},U_{2}}\left(u_{1},u_{2}\right)=\mathcal{N}\left(\mu_{1},\mu_{2},\sigma_{1},\sigma_{2},\rho\right)\]
which is the bivariate Gaussian PDF defined in \eqref{npdf}. By the definition of the bivariate saturation \eqref{bsat},
\[U_1=0 \textrm{ for } U_2<\max\left(-\beta,\alpha\right)\]
Hence, from \eqref{n1e1},
\[
\begin{gathered}
P\left(\alpha-U_{2}<U_{1}<\beta+U_{2}\right)\\
=\int_{\max\left(-\beta,\alpha\right)}^{\infty}\int_{\alpha-u_{2}}^{\beta+u_{2}}\mathcal{N}\left(\mu_{1},\mu_{2},\sigma_{1},\sigma_{2},\rho\right)du_{1}du_{2}\\
=N_1
\end{gathered}\]
\end{proof}

From Theorem \ref{tp}, it can be seen that the probability that the input is \textit{not} saturated is quantified by $N_1$, which in turn, depends on parameters $\mu_2$, $\sigma_2$ and $\rho$, which are exclusive to the multi-variable case, apart from $\mu_{\hat{1}}$ and $\sigma_{\hat{1}}$, which were present even in the single-variable case. One plot, showing the dependence of $N_1$ on actuator noise $\sigma_2$ and the actuator authority $\beta=-\alpha$, is shown in Fig. \ref{fig:N1sigma2}. It can be seen that as the actuator bounds become more variable, the probability that the primary actuator input is saturated approaches 0.5, as proved in the following Theorem. Intuitively, this is because sometimes the actuator bounds increase, allowing for lesser saturation, and sometimes they decrease, resulting in more saturation.

\begin{theorem}
With other parameters fixed, $N_1 \rightarrow 0.5$ as $\sigma_2 \rightarrow \infty$.
\end{theorem}
\begin{proof}
With other parameters fixed, $\sigma_2 \rightarrow \infty \Rightarrow \gamma_1 \rightarrow \infty$ and $\gamma_2 \rightarrow \infty$, as per \eqref{g1} and \eqref{g2}. Thus, $\textrm{erf}\left(\gamma_{1}\right)+\textrm{erf}\left(\gamma_{2}\right)\rightarrow2$. Also, $u'_{2\min} \rightarrow 0$ from \eqref{u2pmin}. Hence,
\[\lim_{\sigma_{2}\rightarrow\infty}N_{1}=\int_{0}^{\infty}\frac{\sqrt{2}}{4\sqrt{\pi}}e^{-\frac{{u'}_{2}^{2}}{2}}\left(2\right)du'_{2}=0.5\]
\end{proof}
\begin{figure}[t]
	\centering
	\includegraphics[width=1\columnwidth]{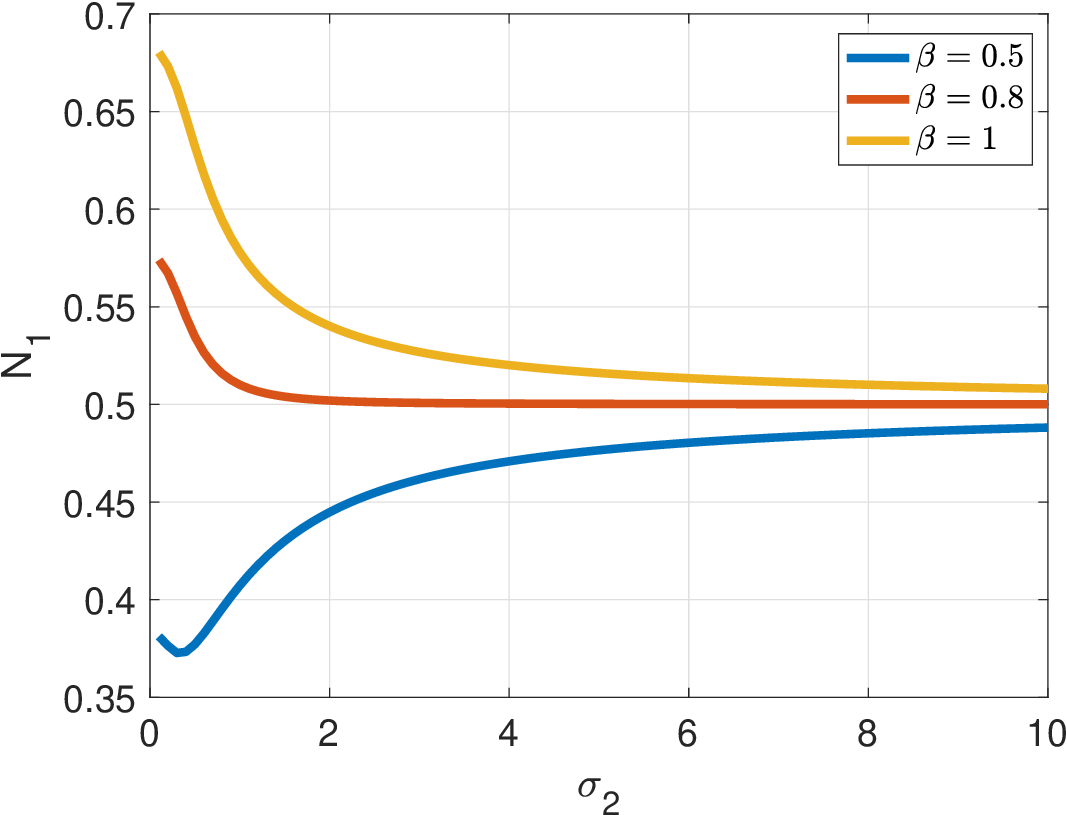}
	\caption{Plot of $N_1$ as a function of $\sigma_2$ for varying $\beta=-\alpha$}
	\label{fig:N1sigma2}
\end{figure}

\subsection{Effect of Correlation between Actuator Inputs}

The primary actuator input $\hat{u}_1(t)$ in Fig. \ref{sl} is correlated to the secondary actuator input $u_2(t)$ due to the feedback provided by the closed loop system. To find the effect of correlation between the actuator inputs, the value of the correlation coefficient $\rho$ was plotted for varying levels of asymmetry and standard deviation of the secondary actuator input. The results are shown in Fig. \ref{fig:e4p1} and Fig. \ref{fig:e5p1}.

\begin{figure}[t]
	\centering
	\includegraphics[width=1\columnwidth]{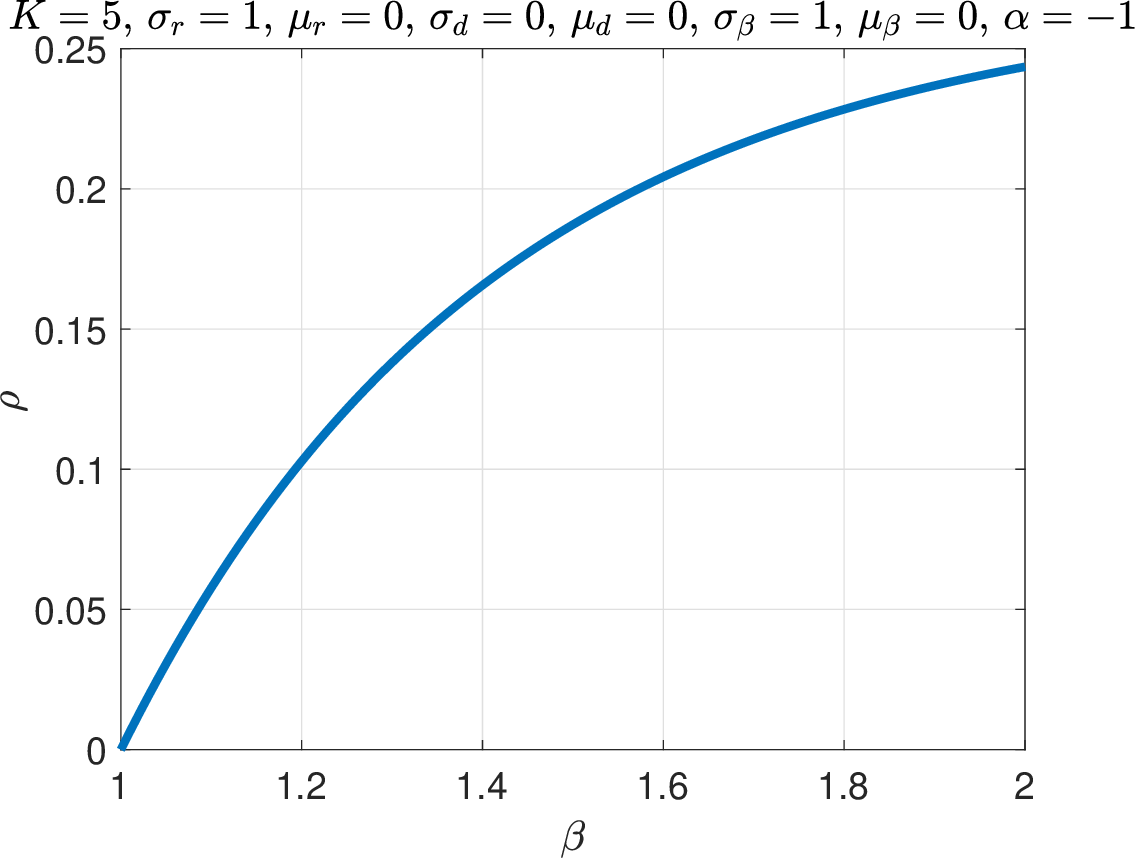}
	\caption{Plot of $\rho$ as a function of asymmetry}
	\label{fig:e4p1}
\end{figure}
\begin{figure}[t]
	\centering
	\includegraphics[width=1\columnwidth]{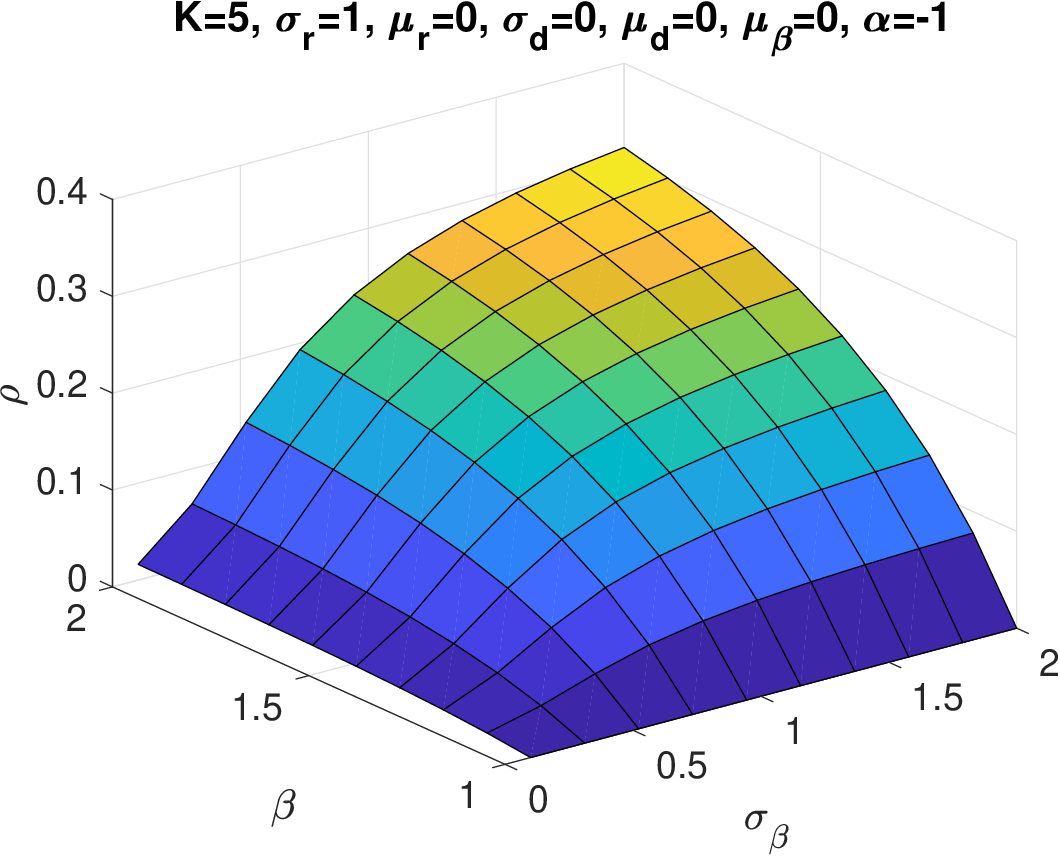}
	\caption{Plot of $\rho$ as a function of asymmetry and noise variability}
	\label{fig:e5p1}
\end{figure}

From the figures it can be seen that as the actuator becomes more asymmetric (i.e., $\beta$, the upper limit of saturation, increases while the lower limit, $\alpha$ is fixed) or when the bounds become more variable (i.e. $\sigma_{2}$ increases), the correlation coefficient increases. However, when the actuator nonlinearity is symmetric ($\alpha=-\beta$), the correlation coefficient is constant for any value of $\sigma_{2}$. This is due to the way $N_2$ and the bivariate saturation function is defined: $u_2(t)$ adds to the upper limit, but subtracts from the lower limit.

\section{Accuracy of Stochastic Linearization}\label{asl}
\begin{figure}[t]
	\centering
	\includegraphics[width=1\columnwidth]{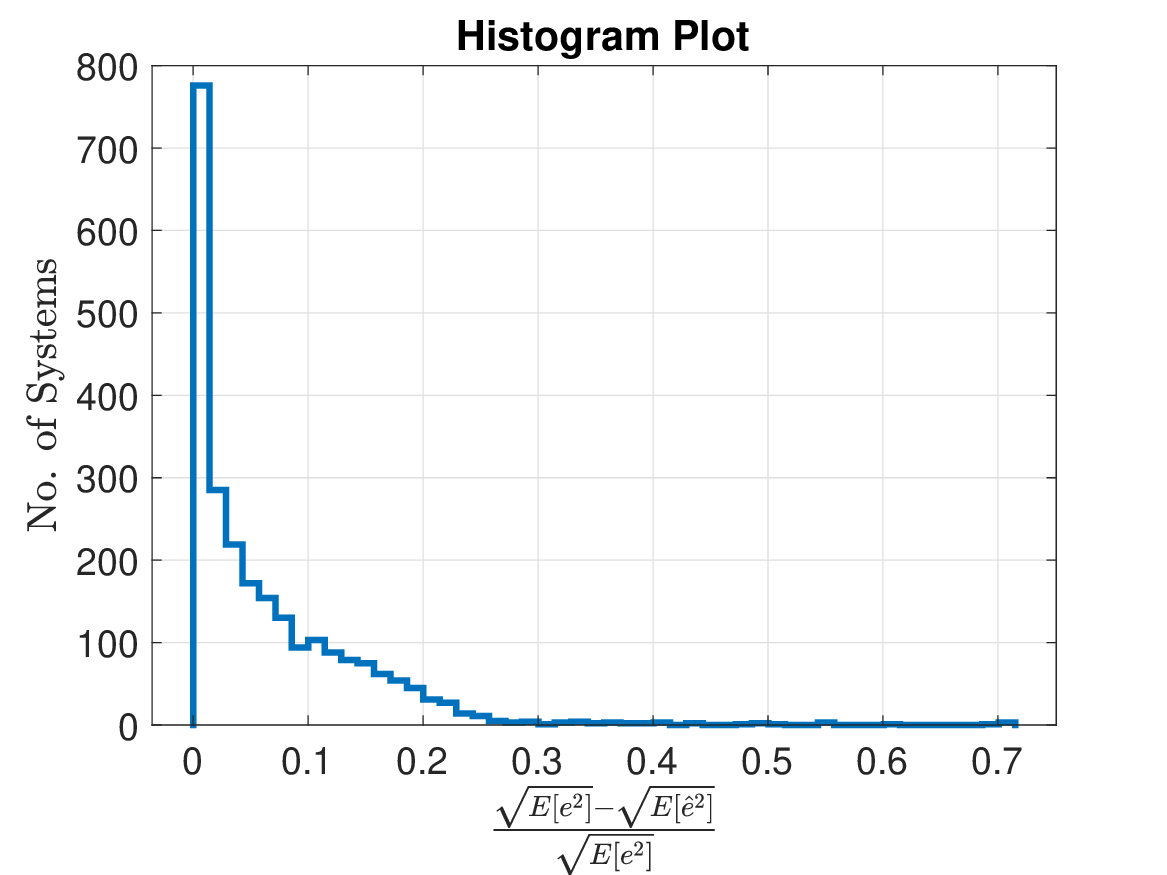}
	\caption{Histogram plot of error accuracy}
	\label{fig:accAO}
\end{figure}
\begin{figure}[t]
	\centering
	\includegraphics[width=1\columnwidth]{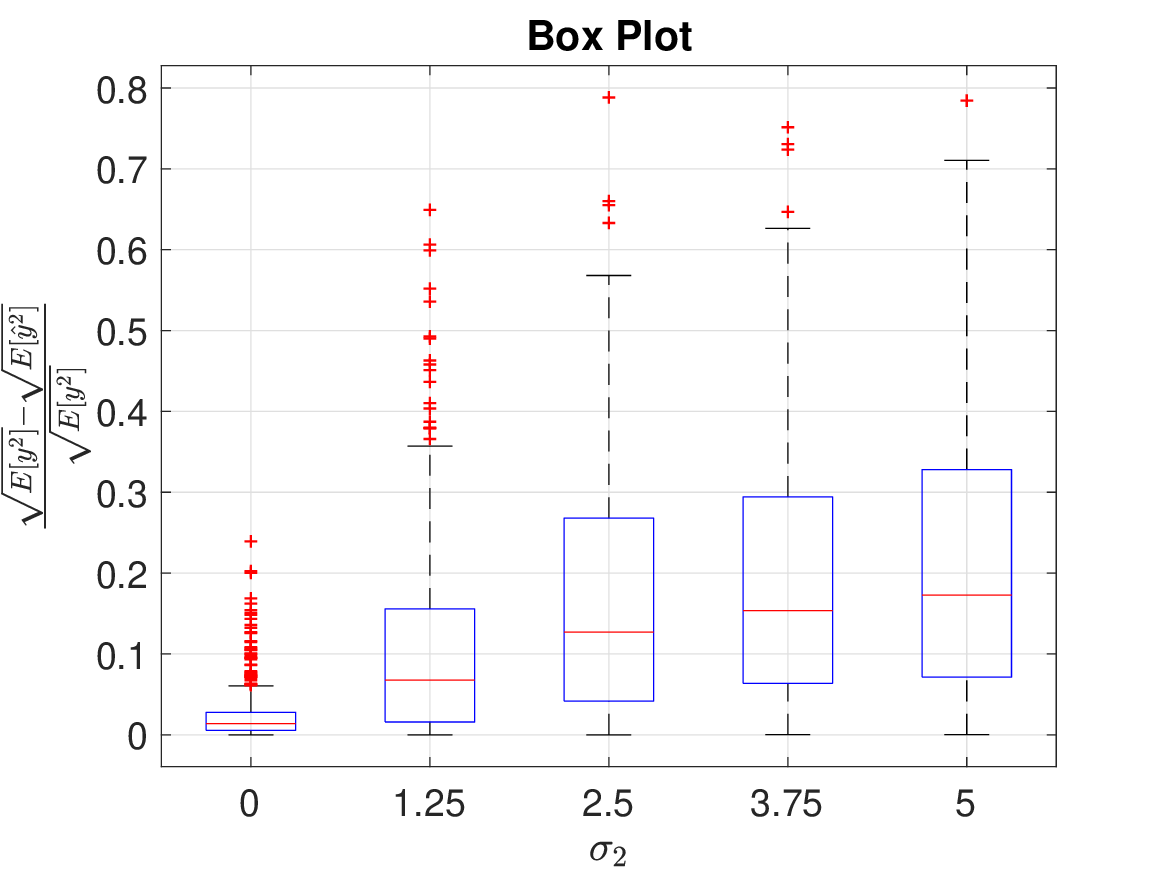}
	\caption{Box plot of output accuracy}
	\label{fig:histAO}
\end{figure}
To investigate the accuracy of the method of stochastic linearization, with focus on the effect of the variability of the actuator noise, a Monte Carlo experiment was performed with the following parameters: 
\[C(s)=K\sim U\left[0.01,50\right]\]
Systems with 2 types of plants were considered. For half of the systems considered,
\[P(s)=\frac{1}{Ts+1},\textrm{ such that }T\sim U\left[0.01,10\right]\]
For the other half,
\[P(s)	=\frac{\omega_{n}^{2}}{s^{2}+2\xi\omega_{n}s+\xi^{2}}\]
such that $\omega_{n}\sim U\left[0.01,10\right], \xi\sim U\left[0.05,2\right]$.
Also, to ensure a fair comparison, the statistical properties of the reference and disturbance signals were assumed to be constant:
\[\mu_{r}=0,\sigma_{r}=1,\mu_{d}=0,\sigma_{d}=1\]
To find the effect of variability in the secondary actuator input $u_2(t)$, the following was assumed about its statistics:
\[\mu_{2}=0, \sigma_{2}\in\left[0,1.25,2.5,3.75,5\right]\]
Finally, the actuator authorities were selected as:
\[\alpha\sim\left[-15,0\right],\beta\sim U\left[0,15\right]\]
3000 systems were considered for simulation, out of which 535 ($\sim 18\%$) unstable systems and those with phase margin $>$ 20 degrees were rejected since they were not practical. All the coloring filters were taken to be of 3rd order Butterworth type, with transfer function \eqref{key}, and cut-off frequency $\Omega=1.43$ kHz, which was found by a separate Monte Carlo experiment to be mean bandwidth of the closed loop systems considered.

\begin{equation}\label{key}
		\small F_{\Omega_{d}}\left(s\right)=\sqrt{\frac{3}{\Omega_{d}}}\left(\frac{\Omega_{d}^{3}}{s^{3}+2\Omega_{d}s^{2}+2\Omega_{d}^{2}s+\Omega_{d}^{3}}\right), \Omega_d=1.43 \textrm{ kHz}
	\end{equation}

The results for accuracy are shown in Figures \ref{fig:accAO} and \ref{fig:histAO}. Fig. \ref{fig:accAO} shows the histogram of the difference in the square root of second moment of the error $e(t)$ in the nonlinear system and the stochastically linearized system, normalized by the square root of second moment of the nonlinear error. It can be seen that stochastic linearization is fairly accurate for most of the systems. Fig. \ref{fig:histAO} shows a box plot of the difference of the square root of second moment of the output between the nonlinear and the stochastically linearized systems, normalized by the square root of second moment of the nonlinear actuator output. It can be seen that as the actuator bounds become more variable, the relative error increases.

\section{Practical Example - Optimal Controller Design}\label{pe}
In this section, a practical example of designing an optimal proportional controller to reduce the standard deviation of the tracking error is presented.

\begin{figure}[t]
	\centering
	\includegraphics[width=1\columnwidth]{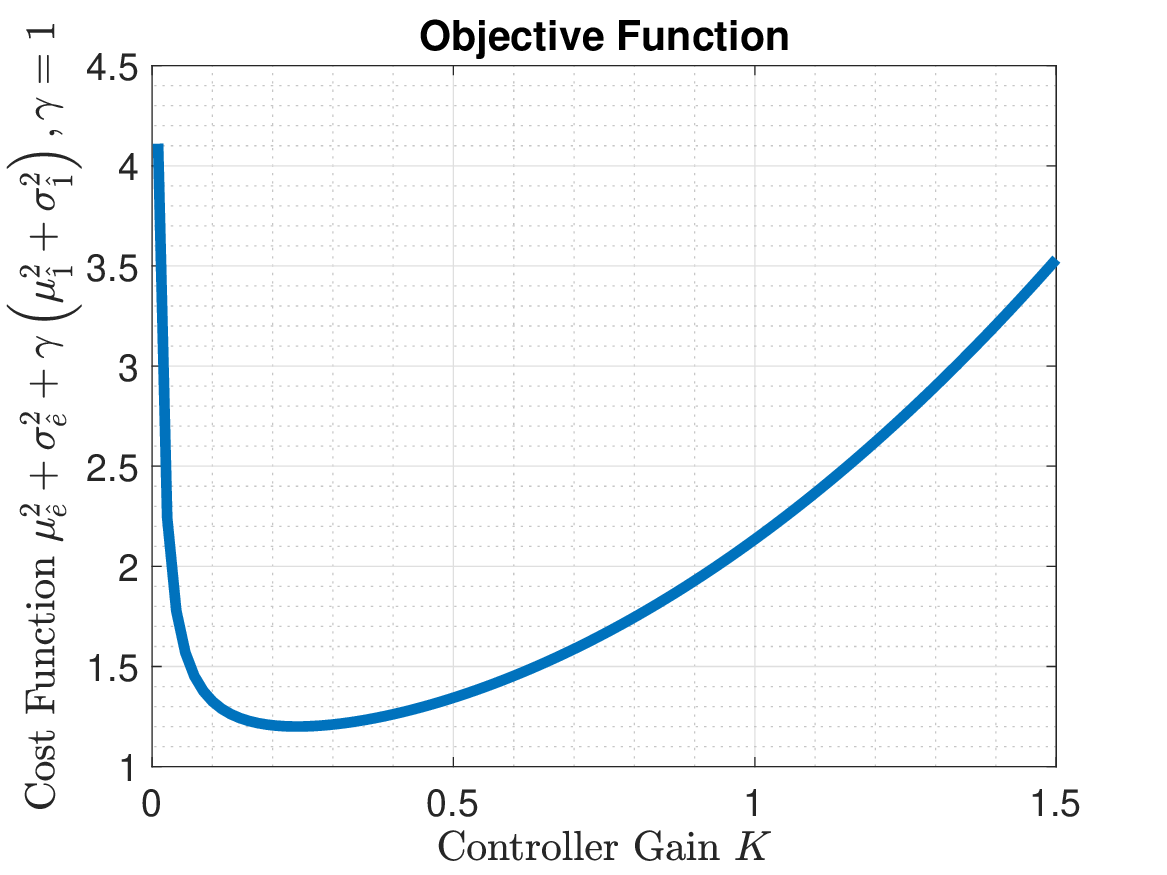}
	\caption{Plot of objective function for system with $P(s) = \frac{10}{s(s+10)}$, $C(s)=~K$, $\alpha=-2$, $\beta=1$, $\mu_1=0$, $\sigma_1=1$, $\mu_2=0$, $\sigma_2=1$ and filter bandwidth 48 rad/s.}
	\label{fig:objF}
\end{figure}

Consider Fig. \ref{fig:Block-Diagram-of}, with $C(s)=D_C=K$ and $P(s) = \frac{10}{s(s+10)}$. The actuator bounds are chosen to be $\alpha=-2$, $\beta=1$, and the input parameters, $\mu_2=0$, $\sigma_2=1$, $\mu_r=0$, $\sigma_r=1$, $\mu_d=0$ and $\sigma_d=1$. The filter bandwidth is chosen to be 48 rad/s, which is close to the system bandwidth. Let the objective function be the sum of the second moment of the tracking error $e(t)$ and that of the primary actuator input $u_1(t)$. It has a well-defined minimum, as seen in Figure \ref{fig:objF}. With this, the following optimization problem can be formulated:

\begin{equation}\label{opt}
\begin{array}{c}
\textrm{min}\;\mu_{\hat{e}}^2+\sigma_{\hat{e}}^2+\gamma\left(\mu_{\hat{1}}^2+\sigma_{\hat{1}}^2\right)\\
\begin{array}{c}
\textrm{subject to}\end{array}\eqref{u2pmin}-\eqref{me2}, \eqref{e1}-\eqref{rho}
\end{array}
\end{equation}
where $\gamma>0$ is a penalty factor, $\sigma_{\hat{1}}$ and $\mu_{\hat{1}}$ are as in \eqref{s1} and \eqref{muu} respectively, and, similarly,
\[\mu_{\hat{e}}=\frac{1}{C_{dc}}\frac{\frac{1}{P_{dc}}\mu_{r}-m-\mu_{d}-N_{2}\mu_{2}}{N_{1}+\frac{1}{C_{dc}}\frac{1}{P_{dc}}}\]
\[\sigma_{\hat{e}}^{2}=\mathbf{C}_{\mathbf{\hat{e}}}\mathbf{\Sigma}\mathbf{C}_{\mathbf{\hat{e}}}^{\mathbf{T}}\] where:
\[\mathbf{C_{\hat{e}}}=\left(\begin{array}{c}
\frac{C_{rf}\sigma_{r}}{D_{c}D_{p}N_{1}+1}\\
-\frac{C_{bf}D_{p}N_{2}\sigma_{2}}{D_{c}D_{p}N_{1}+1}\\
-\frac{C_{df}D_{p}\sigma_{d}}{D_{c}D_{p}N_{1}+1}\\
-\frac{C_{c}D_{p}N_{1}}{D_{c}D_{p}N_{1}+1}\\
-\frac{C_{p}}{D_{c}D_{p}N_{1}+1}
\end{array}\right)^T\]
and $\mathbf{\Sigma}$ is the solution of \eqref{lyap}.

\begin{figure}[t]
	\centering
	\includegraphics[trim={0 0 0 35},width=1.1\columnwidth]{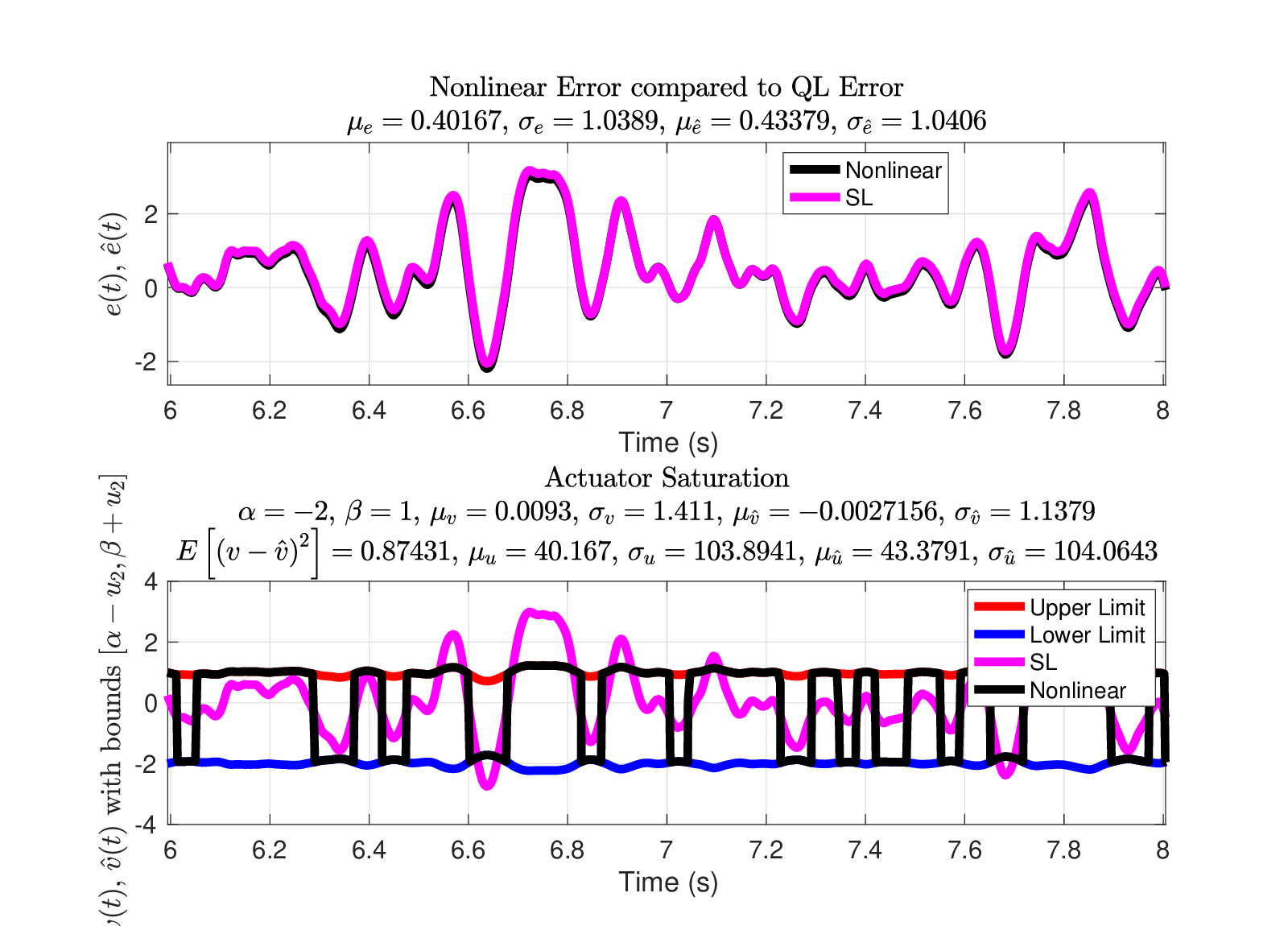}
	\caption{Time series plot for baseline controller.}
	\label{fig:opt1}
\end{figure}

\begin{figure}[t]
	\centering
	\includegraphics[trim={0 0 0 35},width=1.1\columnwidth]{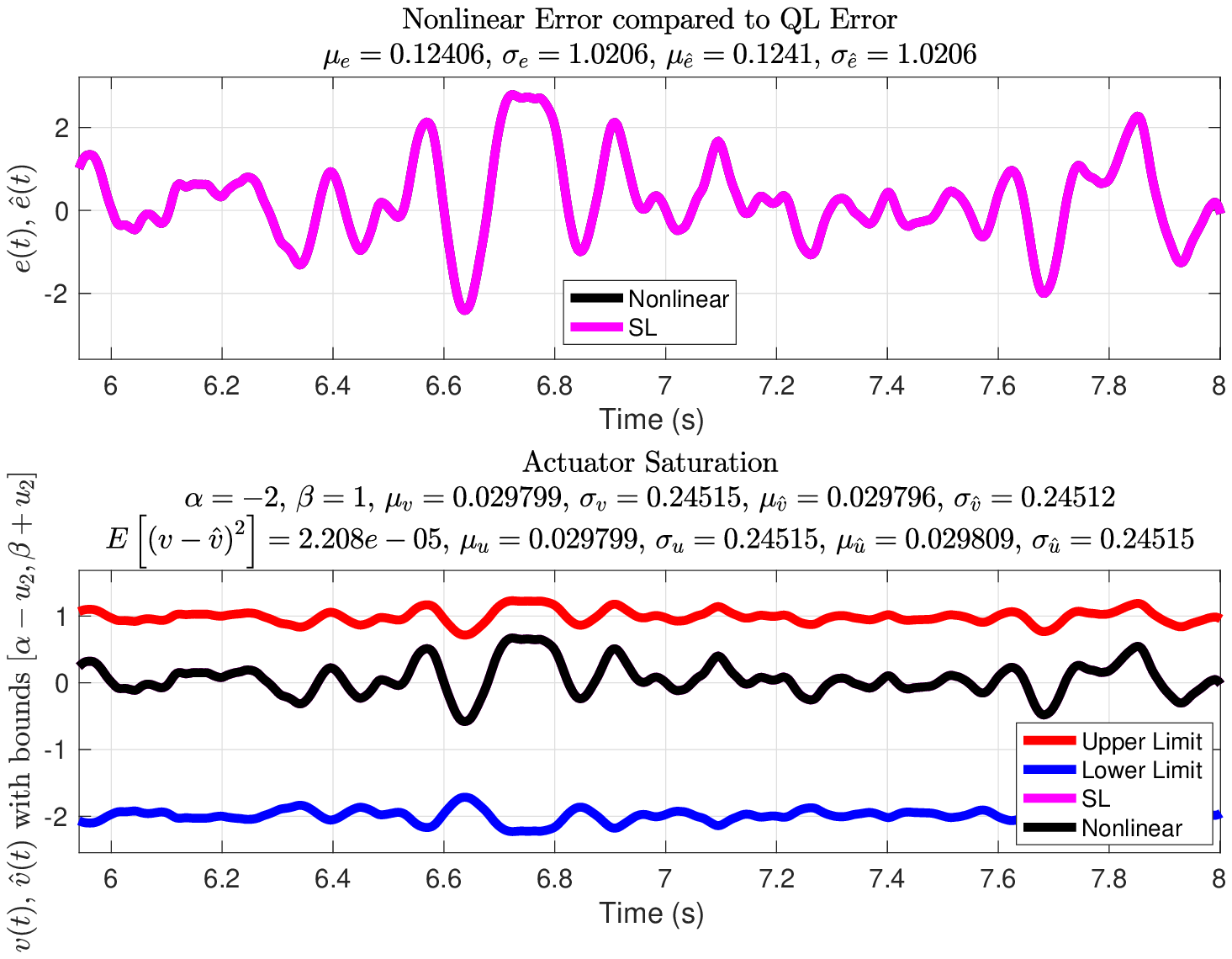}
	\caption{Time series plot for optimal controller}
	\label{fig:opt2}
\end{figure}

The optimization was performed using \texttt{fmincon} of MATLAB\textsuperscript{\textregistered}, with an initial value of $C(s)=100$ and $\gamma=1$. The optimal controller was found to be $C(s)=0.24$. The cost reduced from 1244.5 to 1.2.

The system was simulated in MATLAB/Simulink\textsuperscript{\textregistered} with an initial value of $K=100$, the results of which are shown in Fig. \ref{fig:opt1}. The upper subplot displays the tracking error in the nonlinear system, $e(t)$, and that in the stochastically linearized system, $\hat{e}(t)$. The lower subplot shows the actuator output from the nonlinear system, $v(t)$, and that from the stochastically linearized system, $\hat{v}(t)$, bounded by the actuator limits: $\alpha-u_2(t)$ and $\beta+u_2(t)$. The nonlinear actuator output can be clearly seen to be saturated by the bounds. Fig. \ref{fig:opt2} shows the same system after optimization. It can be seen that the standard deviation of the nonlinear error reduced from 1.04 to 1.02 and its mean reduced from 0.4 to 0.1. Since the optimization also reduced the primary actuator input, it is no longer saturated, and hence, $e(t)$ and $\hat{e}(t)$ coincide.

\section{Conclusion}\label{conc}
In this paper, the theory of stochastic linearization is extended to functions of multiple variables. A general control system with bivariate saturated actuator is considered for analysis, and expressions for equivalent gains and bias have been derived. The accuracy of stochastic linearization was investigated by a Monte Carlo simulation, and was found to be fairly good. Finally, a practical example of optimal control design is presented to show that the method can be used to design optimal controllers.


\appendix \label{appen}
The integral in \eqref{n1e2} can be simplified further to give a closed form expression:

{\small \[
\begin{gathered}
N_{1}	\\
=\int_{u'_{2\textrm{min}}}^{\infty}\frac{\sqrt{2}e^{-\frac{{u'_{2}}^{2}}{2}}}{4\sqrt{\pi}}\left(\mathrm{erf}\left(\frac{\sqrt{2}\left(\mu_{1}-\alpha+\mu_{2}+\sigma_{2}u'_{2}+\rho\sigma_{1}u'_{2}\right)}{2\sigma_{1}\sqrt{1-\rho^{2}}}\right)\right.\\
	\left.+\mathrm{erf}\left(\frac{\sqrt{2}\,\left(\mathrm{\beta}-\mu_{1}+\mu_{2}+\sigma_{2}\,u'_{2}-\rho\,\sigma_{1}\,u'_{2}\right)}{2\sigma_{1}\sqrt{1-\rho^{2}}}\right)\right)du'_{2}\\
	=\frac{1}{2\sqrt{\pi}}\int_{\frac{u'_{2\textrm{min}}}{\sqrt{2}}}^{\infty}e^{-{u''}_{2}^2}\mathrm{erf}\left(K_{1}u''_{2}+K_{2}\right)du''_{2}\\
	+\frac{1}{2\sqrt{\pi}}\int_{\frac{u'_{2\textrm{min}}}{\sqrt{2}}}^{\infty}e^{-{u''}_{2}^2}\mathrm{erf}\left(K_{3}u''_{2}+K_{4}\right)du''_{2}\\\end{gathered}\]}
    where the substitution $\frac{u'_2}{\sqrt{2}}=u''_2$ was done and: 
    \begin{equation}\label{K1}
    K_{1}=\frac{\sigma_{2}+\rho\sigma_{1}}{\sigma_{1}\sqrt{1-\rho^{2}}}
    \end{equation}
    \begin{equation}\label{K2}
    K_{2}=\frac{\mu_{1}-\alpha+\mu_{2}}{\sigma_{1}\sqrt{2\left(1-\rho^{2}\right)}}
    \end{equation}
  	\begin{equation}\label{K3}
  	K_{3}=\frac{\sigma_{2}-\rho\sigma_{1}}{\sigma_{1}\sqrt{1-\rho^{2}}}
  	\end{equation}
    \begin{equation}\label{K4}
    K_{4}=\frac{\beta-\mu_{1}+\mu_{2}}{\sigma_{1}\sqrt{2\left(1-\rho^{2}\right)}}
    \end{equation}
Using the result from \cite{Fayed2014}, the following integral can be defined:
\begin{equation}\label{leq}
\begin{gathered}
L\left(p,a,b\right)=\int_{p}^{\infty}e^{-x^{2}}\textrm{erf}\left(ax+b\right)dx\\
=L_{1}\left(p,a,b\right)+\sum_{n=0}^{\infty}L_{2}\left(p,a,b\right)
\end{gathered}
\end{equation}
such that:
\begin{equation}\label{l1eq}
L_1\left(p,a,b\right)=\frac{\sqrt{\pi}}{2}\left[\textrm{erf}\left(\frac{b}{\sqrt{1+a^{2}}}\right)-\textrm{erf}\left(p\right)\textrm{erf}\left(b\right)\right]
\end{equation}
and
\begin{equation}\label{l2eq}
\begin{gathered}
L_{2}\left(n,p,a,b\right)=e^{-b^{2}}\left\{ \frac{\left(\frac{a}{2}\right)^{2n+1}}{\Gamma\left(n+\frac{3}{2}\right)}\left[1-P_{n+1}\left(p^{2}\right)\right]H_{2n}\left(b\right)\right.\\
\left.+\frac{\textrm{sgn}\left(p\right)\left(\frac{a}{2}\right)^{2n+2}}{\Gamma\left(n+2\right)}P_{n+\frac{3}{2}}\left(p^{2}\right)H_{2n+1}\left(b\right)\right\}
\end{gathered}
\end{equation}
    where:
\[\Phi\left(x\right)=\frac{1}{\sqrt{2\pi}}\intop_{-\infty}^{x}e^{-\frac{t^{2}}{2}}dt\] is the standard univariate normal CDF, 
\[P\left(n,x\right)=\frac{1}{\Gamma\left(n\right)}\intop_{0}^{x}t^{n-1}e^{-t}dt=1-e^{-x}\sum_{j=0}^{n-1}\frac{x^{j}}{j!}\] is the normalized incomplete Gamma function, 
\[\Gamma\left(z\right)=\intop_{0}^{\infty}x^{z-1}e^{-x}dx\] is the Gamma function, and 
\[H_{j}\left(x\right)=j!\sum_{k=0}^{\left[\nicefrac{j}{2}\right]}\frac{\left(-1\right)^{k}}{k!\left(j-2k\right)!}\left(2x\right)^{j-2k}\] is the Hermite polynomial.

For \eqref{leq} to converge, the value of $a$ should satisfy the inequality $\left|a\right|<1$. Hence, $\left|K_{1}\right|<1$ and $\left|K_{3}\right|<1$. This simplifies to:
\begin{equation} \label{rhorange}
\small -\frac{1}{2}\left(\sqrt{2-\left(\frac{\sigma_{2}}{\sigma_{1}}\right)^{2}}-\frac{\sigma_{2}}{\sigma_{1}}\right)<\rho<\frac{1}{2}\left(\sqrt{2-\left(\frac{\sigma_{2}}{\sigma_{1}}\right)^{2}}-\frac{\sigma_{2}}{\sigma_{1}}\right)
\end{equation}
such that $0<\sigma_{2}<\sigma_{1}$.

Fig. \ref{fig:rhoad} shows the admissible values of $\rho$ for which the series will converge.

\begin{figure}[t]
	\centering
	\includegraphics[width=1\columnwidth]{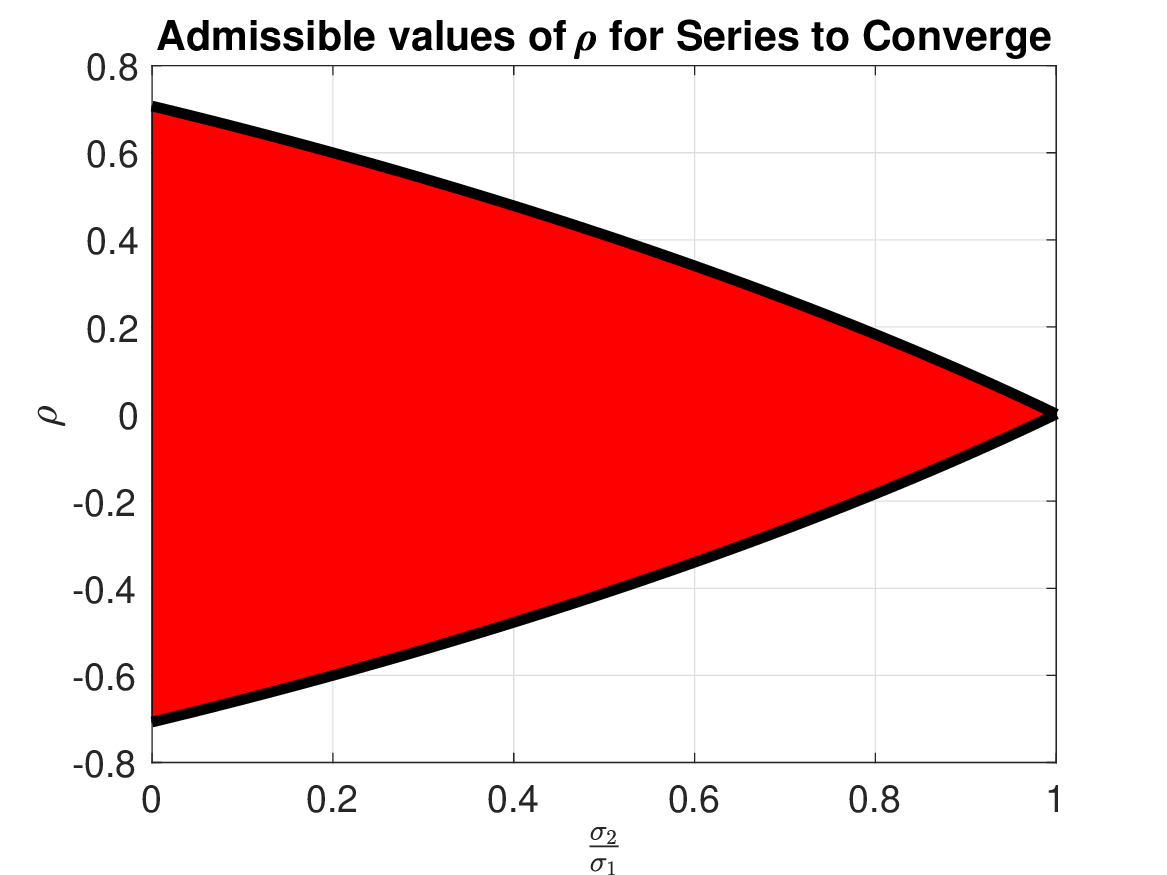}
	\caption{Admissible values for $\rho$}
	\label{fig:rhoad}
\end{figure}
Hence,
    \begin{equation}\label{N1eqf}
    {N_{1}=\frac{1}{2\sqrt{\pi}}\left[L\left(\frac{u'_{2\textrm{min}}}{\sqrt{2}},K_{1},K_{2}\right)+L\left(\frac{u'_{2\textrm{min}}}{\sqrt{2}},K_{3},K_{4}\right)\right]} 
\end{equation}
Similarly,
\begin{equation}\label{N2eqf}
N_{2}=\frac{1}{2\sqrt{\pi}}\left[L\left(\frac{u'_{2\textrm{min}}}{\sqrt{2}},K_{1},K_{2}\right)-L\left(\frac{u'_{2\textrm{min}}}{\sqrt{2}},K_{3},K_{4}\right)\right]
\end{equation}
Similarly,
\begin{equation*}
\begin{gathered}
M=\frac{\sigma_{1}\,\sqrt{1-\rho^{2}}}{\sqrt{2}\,\pi}\int_{\frac{u_{'2\textrm{min}}}{\sqrt{2}}}^{\infty}e^{-{u''}_{2}^{2}}\left({\mathrm{e}}^{-\left(K_{1}u''_{2}+K_{2}\right)^{2}}\right.\\
\left.-{\mathrm{e}}^{-\left(K_{3}u''_{2}+K_{4}\right)^{2}}\right)du''_{2}\\
+\frac{\left(\mu_{1}+\mu_{2}-\alpha\right)}{2\sqrt{\pi}}\int_{\frac{u_{'2\textrm{min}}}{\sqrt{2}}}^{\infty}e^{-{u''}_{2}^{2}}\textrm{erf}\left(K_{1}u''_{2}+K_{2}\right)du''_{2}\\
+\frac{\left(\mu_{1}-\mu_{2}-\beta\right)}{2\sqrt{\pi}}\int_{\frac{u_{'2\textrm{min}}}{\sqrt{2}}}^{\infty}e^{-{u''}_{2}^{2}}\textrm{erf}\left(K_{3}u''_{2}+K_{4}\right)du''_{2}\\
+\frac{\rho\sigma_{1}+\sigma_{2}}{\sqrt{2\pi}}\int_{\frac{u_{'2\textrm{min}}}{\sqrt{2}}}^{\infty}u''_{2}e^{-{u''}_{2}^{2}}\textrm{erf}\left(K_{1}u''_{2}+K_{2}\right)du''_{2}\\
+\frac{\rho\sigma_{1}-\sigma_{2}}{\sqrt{2\pi}}\int_{\frac{u_{'2\textrm{min}}}{\sqrt{2}}}^{\infty}u''_{2}e^{-{u''}_{2}^{2}}\textrm{erf}\left(K_{3}u''_{2}+K_{4}\right)du''_{2}\\
+\frac{1}{2\sqrt{\pi}}\int_{\frac{u_{'2\textrm{min}}}{\sqrt{2}}}^{\infty}\left(\alpha-\mu_{2}-\sqrt{2}\sigma_{2}u''_{2}\right)e^{-{u''}_{2}^{2}}du''_{2}\\
+\frac{1}{2\sqrt{\pi}}\int_{\frac{u_{'2\textrm{min}}}{\sqrt{2}}}^{\infty}\left(\beta+\mu_{2}+\sqrt{2}\sigma_{2}u''_{2}\right)e^{-{u''}_{2}^{2}}du''_{2}
    \end{gathered}
    \end{equation*}
    
Defining:
\begin{equation}
\begin{gathered}
R\left(p,a,b\right)	=\int_{p}^{\infty}e^{-x^{2}}e^{-(ax+b)^{2}}dx\\
	=\frac{\sqrt{\pi}}{2\,\sqrt{a^{2}+1}}{\mathrm{e}}^{-\frac{b^{2}}{a^{2}+1}}\left[1-\mathrm{erf}\left(\frac{pa^{2}+ba+p}{\sqrt{a^{2}+1}}\right)\right]
\end{gathered}
\end{equation}
and:
\begin{equation}
\begin{gathered}
S\left(p,a,b\right)	=\int_{p}^{\infty}xe^{-x^{2}}\textrm{erf}\left(ax+b\right)dx\\
	=\frac{1}{2}\mathrm{erf}\left(a\,p+b\right)\,e^{-p^{2}}\\
	-\frac{a}{2\sqrt{a^{2}+1}}e^{-\frac{b^{2}}{a^{2}+1}}\left[\mathrm{erf}\left(\frac{p\left(a^{2}+1\right)+ba}{\sqrt{a^{2}+1}}\right)+1\right]
\end{gathered}
\end{equation}
we get:
\begin{equation}\label{Meqf}
\small
\begin{gathered}
M	=\frac{\sigma_{1}\sqrt{1-\rho^{2}}}{\sqrt{2}\pi}\left[R\left(\frac{u'_{2\textrm{min}}}{\sqrt{2}},K_{1},K_{2}\right)-R\left(\frac{u'_{2\textrm{min}}}{\sqrt{2}},K_{3},K_{4}\right)\right]\\
	+\frac{\left(\mu_{1}+\mu_{2}-\alpha\right)}{2\sqrt{\pi}}L\left(\frac{u'_{2\textrm{min}}}{\sqrt{2}},K_{1},K_{2}\right)\\
    +\frac{\left(\mu_{1}-\mu_{2}-\beta\right)}{2\sqrt{\pi}}L\left(\frac{u'_{2\textrm{min}}}{\sqrt{2}},K_{3},K_{4}\right)\\
	+\frac{\rho\sigma_{1}+\sigma_{2}}{\sqrt{2\pi}}S\left(\frac{u'_{2\textrm{min}}}{\sqrt{2}},K_{1},K_{2}\right)+\frac{\rho\sigma_{1}
    -\sigma_{2}}{\sqrt{2\pi}}S\left(\frac{u'_{2\textrm{min}}}{\sqrt{2}},K_{3},K_{4}\right)\\
	+\frac{\alpha+\beta}{4}\left[1-\mathrm{erf}\left(\frac{\mathrm{u'_{2min}}}{\sqrt{2}}\right)\right]
\end{gathered}
\end{equation}

Fig. \ref{fig:accSat} shows the accuracy of $N_1$, $N_2$ and $M$ computed using the closed form expressions in \eqref{N1eqf}, \eqref{N2eqf} and \eqref{Meqf} compared to those obtained using the numerical integration using \eqref{n1e1}, \eqref{n2e2} and \eqref{me2}, as a function of the number of terms required in the series part of \eqref{leq}. An algorithm used for calculating $L\left(p,a,b\right)$ in \eqref{N1eqf}, \eqref{N2eqf} and \eqref{Meqf} up to a specified tolerance $tol\%$ in the series expansion is presented below.

\begin{algorithm}[H]
\begin{algorithmic}[1]
\Procedure {L}{$p$, $a$, $b$, $tol$}
\State $partConst \leftarrow $ RHS of \eqref{l1eq}
\State $partSeries \leftarrow $ \Call{L2}{$0$, $p$, $a$, $b$}
\State $pChange \leftarrow \left|\frac{partSeries}{partConst}\right|\times100$
\State $n \leftarrow 0$
\While {$pChange > tol$}
\State $n \leftarrow n+1$
\State $tn \leftarrow $ \Call{L2}{$n$, $p$, $a$, $b$}
\State $pChange \leftarrow \left|\frac{tn}{partSeries+partConst}\right|\times100$
\State $partSeries \leftarrow partSeries + tn$
\EndWhile
\State $result \leftarrow partSeries + partConst$
\State \Return $result$
\EndProcedure
\Statex
\Procedure {L2}{$n$, $p$, $a$, $b$}
\State $y \leftarrow$ RHS of \eqref{l2eq}
\State \Return $y$
\EndProcedure
\end{algorithmic}
\caption{Calculating $L(p,a,b)$ using \eqref{leq}}
\end{algorithm}

\begin{figure}[t]
	\centering
	\includegraphics[width=1\columnwidth]{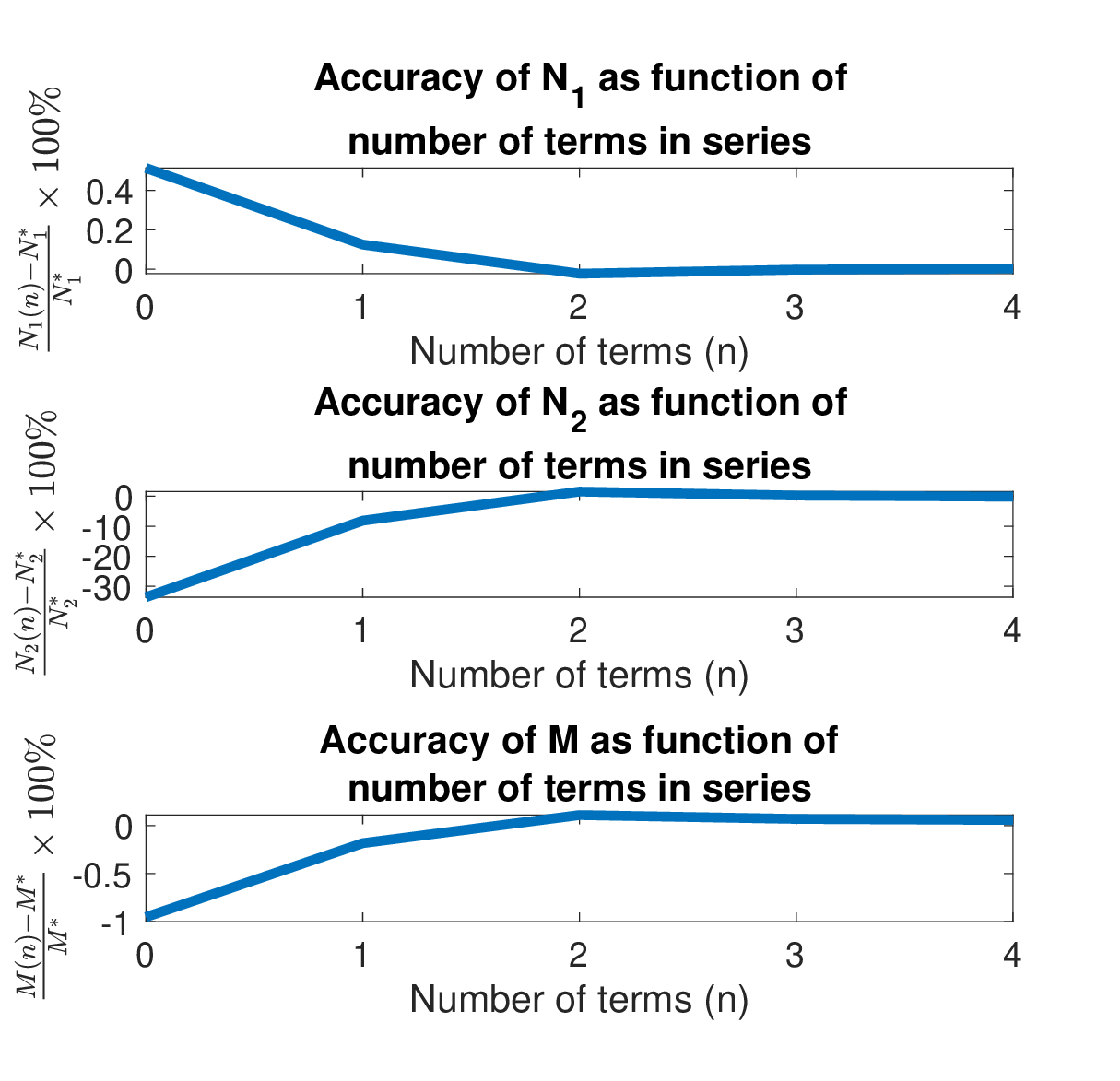}
	\caption{Accuracy of $N_1$, $N_2$ and $M$ computed using series expansion, compared with those obtained using numerical integration. In this example, $\alpha=-3$, $\beta=2$, $\mu_1=1$, $\sigma_1=0.8$, $\mu_2=1$, $\sigma_2=0.7$, $\rho=0.25$. $N_1(n)$, $N_2(n)$ and $M(n)$ refer to the values obtained using the series expansion. $N_1^*$, $N_2^*$ and $M^*$ refer to those obtained using numerical integration.}
	\label{fig:accSat}
\end{figure}



\bibliographystyle{IEEEtran}
\bibliography{IEEEabrv,Mendeley.bib}
%

\end{document}